\documentclass[a4paper,reqno]{amsart}

\usepackage[french,british]{babel}
\usepackage{amssymb,amsmath,amsthm,amscd,mathrsfs,graphicx,color,multicol,booktabs,bm,manfnt} 
\usepackage[cmtip,all,matrix,arrow,tips,curve]{xy}
\usepackage[T1]{fontenc} 
\usepackage{booktabs}

\usepackage{caption}
\usepackage{blindtext}

\usepackage{mathrsfs} 
\usepackage[colorlinks=true,citecolor=blue, urlcolor=blue, linkcolor=blue, pagebackref]{hyperref}
  
\usepackage{mathabx}

\usepackage{manfnt}
\usepackage{centernot}   

\newcommand\restrict[1]{\raisebox{-.5ex}{$|$}_{#1}}

\newcommand{\CC}{{\mathbb C}}
\newcommand{\cA}{{\mathscr A}}

\newcommand{\cB}{{\mathscr B}}

\newcommand{\cE}{{\mathscr E}}
\newcommand{\cF}{{\mathscr F}}
\newcommand{\cG}{{\mathscr G}}
\newcommand{\cH}{{\mathscr H}}

\newcommand{\cK}{{\mathscr K}}
\newcommand{\cL}{{\mathscr L}}
\newcommand{\cM}{{\mathscr M}}

\newcommand{\cO}{{\mathscr O}}
  
\newcommand{\cQ}{{\mathscr Q}}
\newcommand{\cR}{{\mathscr R}}
\newcommand{\cS}{{\mathscr S}}

\newcommand{\cU}{{\mathscr U}}
\newcommand{\cV}{{\mathscr V}} 

\newcommand{\cX}{{\mathscr X}} 
\newcommand{\cY}{{\mathscr Y}}

\newcommand{\dra}{\dashrightarrow}

\newcommand{\FF}{{\mathbb F}}

\newcommand{\GR}{\mathbb{G}\mathrm{r}}

\newcommand{\hra}{\hookrightarrow}

\newcommand{\la}{\langle}

\newcommand{\lra}{\longrightarrow}

\newcommand{\n}{\noindent}
\newcommand{\NN}{{\mathbb N}}
\newcommand{\ov}{\overline}
\newcommand{\PP}{{\mathbb P}}
\newcommand{\qbeck}{{\mathsf q}}
\newcommand{\QQ}{{\mathbb Q}}
\newcommand{\ra}{\rangle}
\newcommand{\RR}{{\mathbb R}}

\newcommand{\sE}{{\mathsf E}}

\newcommand{\wt}{\widetilde}
\newcommand{\ZZ}{{\mathbb Z}}
\newcommand{\Gr}{\mathrm{Gr}}

\theoremstyle{plain} 

\newtheorem{thm}{Theorem}[section]
\newtheorem*{thm*}{Theorem}

\newtheorem{clm}[thm]{Claim}

\newtheorem{crl}[thm]{Corollary}

\newtheorem*{hyp*}{Hypothesis}
\newtheorem{lmm}[thm]{Lemma}
\newtheorem{prp}[thm]{Proposition}
\newtheorem{prp-dfn}[thm]{Proposition-Definition}

\theoremstyle{definition}

\newtheorem{ass-dfn}[thm]{Assumption-Definition}
\newtheorem{dfn}[thm]{Definition}

\theoremstyle{remark}

\newtheorem{expl}[thm]{Example}

\newtheorem*{qst*}{Main Question}
\newtheorem{rmk}[thm]{Remark}

\DeclareMathOperator{\Amp}{Amp}

\DeclareMathOperator{\Blow}{Bl}
\DeclareMathOperator{\ch}{ch}

\DeclareMathOperator{\cl}{cl}

\DeclareMathOperator{\Def}{Def}

\DeclareMathOperator{\divisore}{div}

\DeclareMathOperator{\ext}{ext}
\DeclareMathOperator{\Ext}{Ext}

\DeclareMathOperator{\Hom}{Hom}
\DeclareMathOperator{\Id}{Id}
\DeclareMathOperator{\im}{Im}

\DeclareMathOperator{\NS}{NS}

\DeclareMathOperator{\Pic}{Pic}

\DeclareMathOperator{\pr}{pr}

\DeclareMathOperator{\rk}{rk}

\DeclareMathOperator{\SH}{SH}

\DeclareMathOperator{\Span}{span}

\DeclareMathOperator{\supp}{supp}
\DeclareMathOperator{\Sym}{Sym}
\DeclareMathOperator{\Td}{Td}

\usepackage{multirow}

 \makeindex
 
\begin{document}
 \title{Modular sheaves with many moduli}
 \author{Kieran G. O'Grady}
 \address{Dipartimento di Matematica, 
 Sapienza Universit\`a di Roma,
 P.le A.~Moro 5,
 00185 Roma - ITALIA}
 \email{ogrady@mat.uniroma1.it}
\date{\today}
\thanks{Partially supported by PRIN 2017YRA3LK}
\begin{abstract}
We exhibit moduli spaces of slope stable vector bundles on general polarized HK varieties $(X,h)$ of type $K3^{[2]}$ which have an irreducible component of dimension $2a^2+2$, with $a$ an arbitrary integer greater than $1$. This is done by studying  the case $X=S^{[2]}$ where  $S$ is an elliptic $K3$ surface. 
We show that in this case  there is an irreducible component of the moduli space of stable vector bundles on $S^{[2]}$ which is birational to a moduli space of sheaves on $S$. We expect that if the moduli space of sheaves on $S$ is a smooth HK variety (necessarily of type $K3^{[a^2+1]}$) then 
the following more precise version holds: the closure of the moduli space of slope stable vector bundles on  $(X,h)$  in the moduli space of Gieseker-Maruyama semistable sheaves with its GIT polarization  is  
 a general polarized HK variety of type $K3^{[a^2+1]}$.
\end{abstract}

  \maketitle 
\bibliographystyle{amsalpha}
\section{Introduction}\label{sec:intro} 
\subsection{Background and motivation}\label{subsec:retroterra}
\setcounter{equation}{0}
Starting with Mukai's groundbreaking work of the 80s, moduli of (semistable) sheaves on a (polarized) $K3$ surface have played a prominent r\^ole in Mathematics. These moduli spaces are varieties interesting in themselves (some of them are HK varieties of type $K3^{[n]}$, a few of them admit resolutions which are HK varieties of type OG10), and their Geometry is intertwined with that of the $K3$ surface. One wonders whether moduli of  sheaves on higher dimensional HK varieties may also be the source of interesting Geometry. In~\cite{ogfascimod} we introduced the notion of a modular (torsion free) sheaf. The sheaf  $\cF$ on a HK variety $X$ is modular if $\Delta(\cF)\coloneq -2\rk(\cF)\ch_2(\cF)+\ch_1(\cF)^2$ (the  \emph{discriminant of $\cF$}) satisfies a topological condition (see Subsection~\ref{subsec:ancoramod} for details), for example it is modular if $\Delta(\cF)$ is a multiple of $c_2(X)$. We proved that variation of slope stability for modular sheaves behaves as variation of slope stability for sheaves on surfaces, and that slope (semi)stability of modular sheaves on a HK with a Lagrangian fibration can be tackled  with methods similar to those employed when dealing with sheaves on elliptically fibered $K3$ surfaces. Building on these results, in~\cite{ogfascimod,og-rigidi-su-k3n} we proved existence and uniqueness results analogous to  existence and uniqueness results for stable spherical vector bundles on $K3$ surfaces. In this regard we mention that Mukai's beautiful one-line proof of uniqueness fails,  our unicity argument is substantially more involved - one may view  this  as a foreboding of difficulties to come. 

The main result of the present paper is the following. Let  $(X,h)$ be a general polarized HK variety of type $K3^{[2]}$, with the exclusion of the case in which the divisibility of $h$ is $1$ and $q_X(h)\equiv 2\pmod{8}$. Then for all choices of a positive integer  $a$  (greater than $1$) in an ideal of $\ZZ$ which depends on $(X,h)$ 
there exists a choice of a triple $(r,m,s)\in \NN_{+}\oplus\ZZ\oplus\QQ$  for which the moduli space of $h$ slope stable vector bundles $\cF$ on $X$ with $\rk(\cF)=r$, $c_1(\cF)=mh$ and $\Delta(\cF)=sc_2(X)$ 
contains an irreducible component of dimension $2a^2+2$.

In fact the proof  suggests that the following holds: the moduli spaces that we consider (or to be safe, their closure in the moduli of Gieseker-Maruyama semistable sheaves) are deformations of moduli spaces of sheaves on $K3$ surfaces.  Moreover we expect that in many cases the couple $(\text{mod.~space},\text{GIT pol.})$  is  a general polarized   HK variety of type $K3^{[a^2+1]}$ (or that this holds for a connected component of the moduli space).
At  first glance this appears to be a letdown, but the key word is \emph{general}: we expect to realize a general polarized HK variety of type $K3^{[n]}$ (for certain values of $n$) as a moduli space of sheaves on a 
general polarized HK variety  of type $K3^{[2]}$.
Note that if $n>1$ then a 
general polarized HK variety of type $K3^{[n]}$ cannot be a moduli space of sheaves on a general polarized $K3$ surface because the former has $20$ moduli while the latter has $19$ moduli. In~\cite{gen-hk-as-mod-sheaves} we prove that our expectation is correct (it holds for a connected component of the moduli space, we do not know whether the moduli space is irreducible) when the moduli space has dimension $4$ (the \lq\lq missing case\rq\rq\ $a=1$)
\subsection{Main result}\label{subsec:risulprinc}
\setcounter{equation}{0}
Let $(X,h)$ be a polarized HK variety (polarizations are always primitive). A \emph{mock Mukai vector} is given by   
\begin{equation}\label{fintomukai}
w=(r,l,s)\in \NN_{+}\oplus \NS(X)\oplus H^{2,2}_{\ZZ}(X).
\end{equation}
We  let $M_w(X,h)$ be the moduli space of $h$ slope stable  vector bundles $\cF$ such that 
\begin{equation}
w(\cF)\coloneq (\rk(\cF), c_1(\cF),  \Delta(\cF))=w.
\end{equation}
By Maruyama's classical results,  $M_w(X,h)$ is a scheme of finite type over $\CC$.
\begin{rmk}
 Let $(X,h)$ be a polarized $K3$ surface and let 
$v=(r,l,s)\in  \NN\oplus \NS(X)\oplus H^{2,2}_{\ZZ}(X)$ be a Mukai vector. We denote by $\cM_v(X,h)$ the moduli space of $h$ Gieseker-Maruyama semistable sheaves on $X$ with Mukai vector $v$. If $w$ is the mock Mukai vector 
$(r,l,2r^2+v^2)$, where $v^2=\la v,v\ra$ is the Mukai square of $v$,  then $M_w(X,h)$ is the open subset of $\cM_v(X,h)$ parametrizing slope stable locally free sheaves.
In order to avoid confusion   we  use  the notation $M_w(X,h)$ only if $\dim X>2$.
\end{rmk}
Now suppose that $X$ is of type $K3^{[2]}$. Then $q_X(h)$, the value of the Beauville-Bogomolov-Fujiki quadratic form on $h$,  is a positive even integer. 
The divisibility of $h$, i.e.~the positive generator  of $q_X(h,H^2(X;\ZZ))$, that we denote by $\divisore(h)$, is either $1$ or $2$, and if the latter holds then $q_X(h)\equiv -2\pmod{8}$. In both cases (i.e.~divisibility $1$ and $2$) the corresponding moduli space of polarized varieties is irreducible of dimension $20$. 
 For $a,r_1 \in\NN_{+}$ we let
\begin{equation}\label{nostrodoppiovu}
w:=ar_1\left(2r_1,\frac{2}{\divisore(h)} h,\frac{a r^3_1 c_2(X)}{3}\right).
\end{equation}
\begin{thm}\label{thm:belteo}
Let $r_1$ be a positive integer. 
Let  $(X,h)$ be a polarized  HK variety of type $K3^{[2]}$ such that
\begin{equation}\label{divdipol}
\divisore(h)=
\begin{cases}
1  & \text{if $r_1\equiv 0 \pmod{2}$,} \\
2  & \text{if $r_1\equiv 1 \pmod{2}$,} 
\end{cases}
\end{equation}
and
\begin{equation}\label{congruenze}
q_X(h)\equiv
\begin{cases}
-2 \pmod{2r_1} & \text{if $r_1\equiv 0 \pmod{4}$,} \\
-2r_1-8 \pmod{8r_1}  & \text{if $r_1\equiv 1 \pmod{4}$,} \\
r_1-2 \pmod{2r_1} & \text{if $r_1\equiv 2 \pmod{4}$,} \\
2r_1-8 \pmod{8r_1}   & \text{if $r_1\equiv 3 \pmod{4}$.}
\end{cases}
\end{equation}
Let $a$ be a positive integer greater than $1$ such that $2a$ is a multiple of $r_1$, and let $w$ be the mock Mukai vector given 
by~\eqref{nostrodoppiovu}.  Then the moduli space $M_w(X,h)$ is non empty, and for $(X,h)$  general it has an irreducible component of  dimension $2a^2+2$.
\end{thm}
Following is an outline of the main steps that go into the proof of Theorem~\ref{thm:belteo}. Given sheaves 
$\cE_1$ and $\cE_2$ on a $K3$ surface $S$, there is a natural involution of the sheaf 
$\cE_1\boxtimes\cE_2\oplus \cE_2\boxtimes\cE_1$ on $S^2$ which lifts the involution of $S^2$ exchanging the factors. 
We denote by $\cG(\cE_1,\cE_2)$  the sheaf on $S^{[2]}$ which corresponds via the BKR correspondence to the $\cS_2$-sheaf  $\cE_1\boxtimes\cE_2\oplus \cE_2\boxtimes\cE_1$. In Section~\ref{sec:allagrande} we study basic properties of  the sheaf $\cG(\cE_1,\cE_2)$. (We assume that $\cE_1$ is locally free and $\cE_2$ is torsion free, this guarantees that $\cG(\cE_1,\cE_2)$ is torsion free. If also $\cE_2$ is  locally free then $\cG(\cE_1,\cE_2)$ is locally free.) We show that  $\cG(\cE_1,\cE_2)$ is modular if 
$(\rk(\cE_1),c_1(\cE_1))$ is proportional to 
$(\rk(\cE_2),c_1(\cE_2))$ and the Mukai vectors $v(\cE_1)$, $v(\cE_2)$ are orthogonal. These conditions imply in particular that 
$\ov{v}(\cE_1)^2+\ov{v}(\cE_2)^2=0$, where $\ov{v}(\cE_i)$ is the normalized Mukai vector of $\cE_i$, i.e.~the multiple of the Mukai vector with first entry equal to $1$, and $\ov{v}(\cE_2)^2$ is the square of $\ov{v}(\cE_2)$ in the (rational) Mukai lattice of $S$. 
It follows that there are two types of choices of   $\cE_1,\cE_2$ that might produce a sheaf $\cG(\cE_1,\cE_2)$ which is modular and stable, corresponding to $\ov{v}(\cE_1)^2=\ov{v}(\cE_2)^2=0$ and $\ov{v}(\cE_1)^2<0<\ov{v}(\cE_2)^2$ respectively.  
The first choice gives $v(\cE_1)=v(\cE_2)$  isotropic. This has been considered by Markman
  in~\cite{markman-1-obstructed} and has led to the proof of the analogue of the Shafarevich conjecture for couples of HK varieties of type $K3^{[n]}$, see~\cite{markman-rat-isometries}.  We consider the second choice, i.e.~$v(\cE_1)^2=-2$ and $\cE_1$  is a spherical vector bundle. Thus $v(\cE_2)^2>0$, and in fact $v(\cE_2)^2$ may be arbitrarily large.   A key fact  that holds with some mild assumptions is that all (nearby) deformations of $\cG(\cE_1,\cE_2)$ are given by $\cG(\cE_1,\cE'_2)$
  where $\cE_2'$ is a (nearby)  deformation of $\cE_2$. Now suppose that $\cE_2$ is stable 
 and  $\cG(\cE_1,\cE_2)$ is slope stable. Then we get that as 
  $\cE'_2$ varies among stable deformations of $\cE_2$ the sheaves $\cG(\cE_1,\cE'_2)$ fill out an irreducible component of a moduli space of sheaves on $S^{[2]}$ (and $\cG(\cE_1,\cE'_2)$ is locally free for a general such $\cE_2'$). Actually one checks easily that we get a component birational to the relevant moduli space of (semistable) sheaves on $S$. But we are getting ahead of ourselves: the proof that this idea works is in Section~\ref{sec:modbir}.
Section~\ref{sec:eccoesempi} lists examples of couples   $\cE_1,\cE_2$ such that  $\cG(\cE_1,\cE_2)$ is modular. 
In that section we also show that with a suitable choice of  $\cE_1,\cE_2$ the couple $(S^{[2]},\cG(\cE_1,\cE_2))$ deforms to 
the couple $(F(Y),\bigwedge^2 \cQ)$ where $F(Y)$ is the variety of lines on a general cubic forufold $Y\subset\PP^5$ and $\cQ$ is the restriction to $F(Y)$ of the tautological rank $4$ quotient vector bundle on $\Gr(1,\PP^5)$. These examples of modular vector bundles were discovered by Fatighenti, see~\cite{fatighenti-esempi}. The present work has been motivated by the desire to understand Fatighenti's example. In Section~\ref{sec:nonatomico} we perform more computations in order to determine whether $\cG(\cE_1,\cE_2)$ is atomic or not (of course we assume that it is modular). The answer is that it is atomic if and only if the Mukai vectors $v(\cE_1)$, $v(\cE_2)$ are both isotropic. Section~\ref{sec:ancorastab} extends  the results on  variation of slope (semi)stability of modular sheaves on a HK variety with respect to ample classes proved in~\cite{ogfascimod} to variation with respect to  K\"ahler classes. In the same section we give results on slope (semi)stability of modular sheaves on a HK variety with a Lagrangian fibration which go beyond those proved in loc.~cit. They are needed in order to deal with stability of  a sheaf on a  Lagrangian fibration which restricts to a strictly semistable sheaf on a general Lagrangian fiber.
As mentioned above, in Section~\ref{sec:modbir} we prove that under certain hypotheses the sheaf $\cG(\cE_1,\cE_2)$ is slope stable by applying some of the results in Section~\ref{sec:ancorastab} together with results on certain sheaves on   elliptic $K3$ surfaces which are proved in Appendix~\ref{sec:sheavesonellk3}.
Theorem~\ref{thm:belteo} is proved in Section~\ref{sec:alfinlaprova}. We show that a general sheaf $\cG(\cE_1,\cE_2)$ in the irreducible component of the moduli space of sheaves on $S^{[2]}$ described above extends to a nearby deformation of $(S^{[2]},c_1(\cG(\cE_1,\cE_2))$ by applying Verbitsky's fundamental results on projectively hyperholomorphic vector bundles together with the results  of Section~\ref{sec:ancorastab}.
\subsection{Acknowledgements}
As mentioned above, this work originated from Fatighenti's example, I thank Enrico heartily. Thanks also go to Alesssio Bottini for several conversations on arguments related to this paper.
\section{Simple modular sheaves with many moduli}\label{sec:allagrande}
 \subsection{A construction of sheaves on $S^{[2]}$}\label{subsec:fasciogi}
\setcounter{equation}{0}
Let $S$ be a (complex) smooth projective surface, and let $X_2(S)$ be the blow up of $S^2$ along the diagonal. We have a commutative diagram
\begin{equation}\label{commiso}
\xymatrix{ X_2(S)\ar[d]_{\rho}\ar[rr]^{\tau}    &  &  S^2 \ar[d]^{\pi}\\ 
  S^{[2]}  \ar[rr]^{\gamma} & & S^{(2)} }
\end{equation}
where $\pi$ is the quotient map and $\gamma$ is the cycle (or Hilbert-to-Chow) map. 
The map $\rho$ in~\eqref{commiso} is finite,  flat, of degree $2$.
Let $\pr_i\colon S^2\to S$ be the $i$-th projection, and 
let $\tau_i\colon X_2(S)\to S$ be the composition $\tau_i:=\pr_i\circ \tau$. 

Given sheaves $\cE_1,\cE_2$ on $S$ let $\cF=\cF(\cE_1,\cE_2)$ be the  sheaf on $X_2(S)$ defined by
\begin{equation}
\cF(\cE_1,\cE_2):=\tau_1^{*}\cE_1\otimes\tau_2^{*}\cE_2\oplus \tau_1^{*}\cE_2\otimes\tau_2^{*}\cE_1.
\end{equation}
Let $\sigma$ be the involution of $X_2(S)$ which lifts the involution of $S^2$ which exchanges the factors. Thus $\tau_{3-i}\circ\sigma=\tau_i$ for $i\in\{1,2\}$. The obvious isomorphism $\sigma^{*}\cF(\cE_1,\cE_2)\cong \cF(\cE_1,\cE_2)$ defines an action 
of the  symmetric group $\cS_2$  on $\cF(\cE_1,\cE_2)$ which is compatible with its action on $X_2(S)$.
  Since the action of $\cS_2$ on  $X_2(S)$  maps  any fiber of  $\rho$ to itself,  we get
 an action of  $\cS_2$  on $\rho_{*}(\cF)$ (i.e.~an action lifting the trivial action on $S^{[2]}$). 
\begin{dfn}
Let $\cG=\cG(\cE_1,\cE_2)= \rho_{*}(\cF)^{\cS_2}$ be the sheaf of $\cS_2$-invariants for the action 
of  $\cS_2$  on $\rho_{*}(\cF)$.
\end{dfn}
By definition of the $\cS_2$-action we have
\begin{equation}\label{tuttoqua}
\cG(\cE_1,\cE_2)\cong \rho_{*}(\tau_1^{*}\cE_1\otimes\tau_2^{*}\cE_2).
\end{equation}
The remark below explains why we define $\cG(\cE_1,\cE_2)$ as a sheaf of $\cS_2$-invariants.

\begin{rmk}
Suppose that $S$ is a $K3$ surface. The Bridgeland-King-Reid (BKR) McKay correspondence~\cite{bkr} applied to the category $D_{\cS_2}(S^2)$
of $\cS_2$-equivariant (coherent) sheaves on $S^2$  gives an equivalence between $D_{\cS_2}(S^2)$  and  the category of
(coherent) sheaves on $S^{[2]}$.
Let $\ov{\cF}(\cE_1,\cE_2)$ be the $\cS_2$-equivariant sheaf on $S^2$ defined by
\begin{equation}
\ov{\cF}=\ov{\cF}(\cE_1,\cE_2):=\pr_1^{*}\cE_1\otimes\pr_2^{*}\cE_2\oplus \pr_1^{*}\cE_2\otimes\pr_2^{*}\cE_1.
\end{equation}
If we adopt the definition in~\cite[Section~2.4]{krug-bkr} then 
$\cG(\cE_1,\cE_2)$ corresponds to  $\ov{\cF}(\cE_1,\cE_2)$ via the BKR McKay correspondence.
\end{rmk}
We discuss a few properties of the above construction. First note that
\begin{equation}
\cG(\cE_1,\cE_2'\oplus\cE_2'')\cong 
\cG(\cE_1,\cE_2')\oplus \cG(\cE_1,\cE_2'').
\end{equation}
Secondly we discuss the case  $\cE_1=\cE_2=\cA$ where  $\cA$ is  locally free. Following~\cite{ogfascimod} we associate to $\cA$ locally free sheaves $\cA[2]^{\pm}$ on $S^{[2]}$ as follows.  Let
\begin{equation}\label{stanlaurel}
\phi\colon \sigma^{*}\left(\tau_1^{*}(\cA)\otimes\tau_2^{*}(\cA)\right)  \overset{\sim}{\lra}  \tau_1^{*}(\cA)\otimes\tau_2^{*}(\cA) 
\end{equation}
be the isomorphism switching the factors of the tensor products. Then $\phi$ defines an action of $\cS_2$ on $\tau_1^{*}(\cA)\otimes\tau_2^{*}(\cA)$ which is compatible with its action on $X_2(S)$. Hence we get an $\cS_2$-action on 
$\rho_{*}\left(\tau_1^{*}(\cA)\otimes\tau_2^{*}(\cA)\right)^{\cS_2}$. 
The sheaf of $\cS_2$ invariants of 
$\rho_{*}\left(\tau_1^{*}(\cA)\otimes\tau_2^{*}(\cA)\right)^{\cS_2}$ for this  action is $\cA[2]^{+}$. 
One may define another action of $\cS_2$ multiplying $\phi$ by $-1$. The sheaf of $\cS_2$ invariants of 
$\rho_{*}\left(\tau_1^{*}(\cA)\otimes\tau_2^{*}(\cA)\right)^{\cS_2}$ for this second  action (i.e.~anti-invariants of the first action) is $\cA[2]^{-}$. 
\begin{prp}\label{prp:casoaa}
Let $\cA$ be a locally free sheaf on $S$. Then 
\begin{equation}\label{aaspezza}
\cG(\cA,\cA)\cong \cA[2]^{+}\oplus \cA[2]^{-}.
\end{equation}
\end{prp}
\begin{proof}
We have injections 
\begin{equation*}
\begin{matrix}
\tau_1^{*}(\cA)\otimes\tau_2^{*}(\cA) & \hra & \tau_1^{*}(\cA)\otimes\tau_2^{*}(\cA)\oplus \tau_1^{*}(\cA)\otimes\tau_2^{*}(\cA)=\cF(\cA,\cA)\\
\xi & \mapsto & (\xi,\xi)
\end{matrix}
\end{equation*}
and
\begin{equation*}
\begin{matrix}
\tau_1^{*}(\cA)\otimes\tau_2^{*}(\cA) & \hra & \tau_1^{*}(\cA)\otimes\tau_2^{*}(\cA)\oplus \tau_1^{*}(\cA)\otimes\tau_2^{*}(\cA)=\cF(\cA,\cA)\\
\xi & \mapsto & (\xi,-\xi)
\end{matrix}
\end{equation*}
Moreover $\cF(\cA,\cA)$ splits as the direct sum of the images of the two injections.  The first of the above maps 
is $\cS_2$-equivariant if the action on $\tau_1^{*}(\cA)\otimes\tau_2^{*}(\cA) $ is the first one defined above, and the second one is $\cS_2$-equivariant if the action on $\tau_1^{*}(\cA)\otimes\tau_2^{*}(\cA) $ is the second one defined above.
 Taking $\cS_2$ invariants of the direct images for $\rho_{*}$ one gets the isomorphism in~\eqref{aaspezza}.
\end{proof}
\begin{prp}
With notation as above, the following hold:
\begin{enumerate}
\item
If $\cE_1,\cE_2$ are locally free, then $\cG(\cE_1,\cE_2)$ is locally free.
\item
If $\cE_1$ is locally free, and $\cE_2$ is torsion free,  then $\cG(\cE_1,\cE_2)$ is 
torsion free.
\end{enumerate}
\end{prp}
\begin{proof}
(1):  Since $\cF=\cF(\cE_1,\cE_2)$ is a tensor product of locally free sheaves, it is locally free.  
Since the map $\rho$ in~\eqref{commiso} is finite  and flat,  it follows that $\rho_{*}(\cF)$ is 
 locally free. Thus $\cG= \rho_{*}(\cF)^{\cS_2}$ is locally free.

\n
(2):  Since $\cE_2$ is a torsion free sheaf on a (smooth) surface it has a two-step locally free resolution which is an injection of vector bundles away from a subset of codimension $2$. Pulling back via $\tau_2$ to $X_2(S)$ and tensoring by $\tau_1^{*}\cE_1$ we get that $\cF(\cE_1,\cE_2)$ 
has a two-step locally free resolution which is an injection of vector bundles away from a subset of codimension $2$. It follows that 
$\cF(\cE_1,\cE_2)$ is torsion free. 
Hence $\rho_{*}\cF(\cE_1,\cE_2)$ is torsion free, and a fortiori the subsheaf $\cG(\cE_1,\cE_2)$ is torsion free.
\end{proof}
Next we do the construction in families. Let $B$ be a scheme, and let $\sE_1,\sE_2$ be sheaves on $S\times B$. 
We define a sheaf $\cF(\sE_1,\sE_2)$ on $X_2(S)\times B$ by letting
\begin{equation*}
\cF(\sE_1,\sE_2):=(\tau_1\times\Id_B)^{*}\sE_{1}\otimes
(\tau_2\times\Id_B)^{*}\sE_{2}\oplus 
(\tau_1\times\Id_B)^{*}\sE_{2}\otimes
(\tau_2\times\Id_B)^{*}\sE_{1}.
\end{equation*}
The symmetric group $\cS_2$ acts on $(\rho\times B)_{*}\cF(\sE_1,\sE_2)$: we let 
$\cG(\sE_1,\sE_2)$ be the subsheaf of $\cS_2$-invariants. 
\begin{prp}\label{prp:gipiatto}
With notation as above, suppose 
that $\sE_1$ is a $B$-flat  family of locally free sheaves on $S$, and that
$\sE_2$ is a $B$-flat  family of torsion free sheaves sheaves on $S$.
Then $\cG(\sE_1,\sE_2)$ is a $B$-flat family of torsion free on $S^{[2]}$. 
\end{prp}
\begin{proof}
Since $\sE_2$ is a $B$-flat family of torsion free sheaves on $S$, it has a two-step locally free resolution 
\begin{equation}\label{fijirugby}
0\lra \sE_2^1\lra \sE_2^0\lra \sE_2\lra 0
\end{equation}
which restricts to a locally free resolution of $\sE_{2| S\times\{b\}}$ for every (schematic) point of $B$. The pull-back to $X_2(S)\times B$ of the exact sequence in~\eqref{fijirugby}  remains exact. It follows that  $\cF(\sE_1,\sE_2)$ is 
the direct sum of two sheaves, each of which has a two-step locally free resolution which remains exact when restricted to each fiber of the projection $X_2(S)\times B\to B$. This implies that $\cF(\sE_1,\sE_2)$ is $B$-flat, see Theorem~22.5 in~\cite{matsumura}. Since $\rho$ is flat we get that $(\rho\times\Id)_{*}\cF(\sE_1,\sE_2)$ is $B$-flat, and hence also its sheaf of $\cS_2$-invariants.

\end{proof}
\subsection{Main results}\label{subsec:queer}
\setcounter{equation}{0}
Notation is as in Subsection~\ref{subsec:fasciogi}. Below are the main results of the present section.
\begin{prp}\label{prp:extdigi}
Suppose
that    $S$ is a $K3$ surface, that   
$\cE_1,\cE_2$ are simple sheaves on $S$, that $\cE_1$ is locally 
free, and that $\cE_2$  is torsion-free.  Then the following hold:
\begin{enumerate}
\item
The sheaf  $\cG=\cG(\cE_1,\cE_2)$ is  simple 
if and only if
\begin{equation}\label{noninter}
\ext^0_S(\cE_1,\cE_2)\cdot\ext^0_S(\cE_2,\cE_1)=0.
\end{equation}
\item
If $\cG$ is  simple, then  
\begin{equation*}
\ext^1_{S^{[2]}}(\cG,\cG)  =  \ext^1_S(\cE_1,\cE_1)+\ext^1_S(\cE_2,\cE_2)+
(\ext^0_S(\cE_1,\cE_2)+\ext^0_S(\cE_2,\cE_1))\cdot\ext^1(\cE_1,\cE_2).
\end{equation*}
\item
If 
\begin{equation}\label{duezeri}
\ext^0_S(\cE_1,\cE_2)=\ext^0_S(\cE_2,\cE_1)=0
\end{equation}
(in particular $\cG$ is  simple), then 
\begin{equation}\label{noodles}
\ext^2_{S^{[2]}}(\cG,\cG)  = 
2+\ext^1_S(\cE_1,\cE_2)^2+\ext^1_S(\cE_1,\cE_1)\cdot \ext^1_S(\cE_2,\cE_2),
\end{equation}
\item
If~\eqref{duezeri} holds, then deformations of $\cG$ are unobstructed, and $\Def(\cG)$ is identified with 
$\Def(\cE_1)\times\Def(\cE_2)$ via the map
\begin{equation}\label{defoprod}
\Def(\cE_1)\times \Def(\cE_2)\overset{\Phi}{\lra}  \Def(\cG)
\end{equation}
which associates to deformations  $\cE_1(s)$, $\cE_2(t)$ of $\cE_1,\cE_2$ respectively the sheaf $\cG(\cE_1(s),\cE_2(t))$ (this makes sense by Proposition~\ref{prp:gipiatto}).
\end{enumerate}
\end{prp}
In order to state the other main result of the present section we recall the description of the second cohomology of $S^{[2]}$. 
Let  $E\subset X_2(S)$ be the  exceptional 
divisor of the blow up map $\tau\colon X_2(S)\to S^2$, and let   
\begin{equation}
e:=\cl(E)\in H^2(X_2(S),\ZZ).
\end{equation}
There exist a homomorphism
\begin{equation}
\mu\colon H^2(S;\ZZ)\lra H^2(S^{[2]};\ZZ)
\end{equation}
and a class $\delta\in H^2(S^{[2]};\ZZ)$ such that 
\begin{equation}\label{tiromu}
\rho^{*}\mu(\alpha)=\tau_1^{*}\alpha+\tau_2^{*}\alpha,\quad \rho^{*}\delta=e,
\end{equation}
where $\alpha$ is an arbitrary class in $H^2(S;\ZZ)$. We have a direct sum decomposition
\begin{equation}
H^2(S^{[2]};\ZZ)=\im(\mu)\oplus\ZZ\delta,
\end{equation}
whose addends are orthogonal for the Beauville-Bogomolov-Fujiki (BBF) symmetric bilinear form $q_{S^{[2]}}$. Moreover $ q_{S^{[2]}}(\mu(\alpha))=\alpha^2$ for $\alpha\in H^2(S;\ZZ)$, and 
$q_{S^{[2]}}(e)=-2$. As a matter of notation we denote by the same symbol $\mu$ the extension of $\mu$ to a linear map $H^2(S;\CC)\to H^2(S^{[2]};\CC)$.

\begin{prp}\label{prp:gimodulare}
Let    $S$ be a $K3$ surface. 
Let $\cE_1,\cE_2$ be torsion-free sheaves on $S$ of ranks $r_1,r_2$, with $\cE_1$  locally 
free. Let 
$\cG=\cG(\cE_1,\cE_2)$.
\begin{enumerate}
\item 
We have
\begin{equation}\label{rangoeciuno}
\rk(\cG)= 2r_1 r_2 ,\quad c_1(\cG) =r_2  \mu(c_1(\cE_1))+r_1  \mu(c_1(\cE_2))-r_1 r_2\delta.
\end{equation}
\item 
Suppose in addition that
\begin{equation}\label{sameslope}
r_2 \cdot c_1(\cE_1)=r_1 \cdot c_1(\cE_2),
\end{equation}
and that, letting   $v(\cE_i)$ be the Mukai vector of $\cE_i$, we have
\begin{equation}\label{sommaquad}
 r_2^2 \cdot v(\cE_1)^2+r_1^2 \cdot v(\cE_2)^2=0,
\end{equation}
where $ v(\cE_i)^2\coloneq \la v(\cE_i),v(\cE_i)\ra$ is the square of the Mukai pairing $\la \cdot,\cdot\ra$.
Then  
\begin{equation}\label{discform}
\Delta(\cG)  =  \frac{ r_1^2 r_2^2}{3} c_2(S^{[2]}).
\end{equation}
\end{enumerate}
\end{prp}
\begin{rmk}
The most important result in Proposition~\ref{prp:gimodulare} is the assertion that, under suitable hypotheses the discriminant of 
$\cG(\cE_1,\cE_2)$ is a multiple of $c_2(S^{[2]})$, 
i.e.~that $\cG(\cE_1,\cE_2)$  is a modular sheaf.  In Subsection~\ref{subsec:ancoramod} we recall the definition of modular sheaf. The key features of modular sheaves are the following. First variation of stability behaves as if the hyperk\"ahler variety were a surface. Secondly, one may relate stability of a modular sheaf on a hyperk\"ahler with a Lagrangian fibration  and (semi)stability of its restriction to a general Lagrangian fiber, provided the polarization is sufficiently close to the boundary of the ample cone which corresponds to the Lagrangian fibration. These results, which are presented in Section~\ref{sec:ancorastab}, provide the theoretical basis of our proof of the main result, i.e.~Theorem~\ref{thm:belteo}.
\end{rmk}
\begin{rmk}\label{rmk:vunovudue}
Assume that the equality in~\eqref{sameslope} holds. Then  the equality in~\eqref{sommaquad} holds if and only if
\begin{equation}
\la v(\cE_1),v(\cE_2)\ra=0.
\end{equation}
Note that this is equivalent to the condition $\chi_S(\cE_1,\cE_2)=0$.
\end{rmk}
\begin{rmk}\label{rmk:cosgen}
Let $S$ be a $K3$ surface. The construction in Subsection~\ref{subsec:fasciogi} extends to $S^{[n]}$ and  it does  give modular sheaves under suitable hypotheses. 
Let $\cE_1,\ldots,\cE_n$ be  sheaves on $S$, locally free, with the possible exception of one which is torsion-free. 
Let $X_n(S)$ be the $n$-th isospectral scheme of $S$  (see Definition 3.2.4 in~\cite{haiman}), with maps 
$\tau\colon X_n(S)\to S^n$ (the blow-up of the big diagonal) and $\rho\colon X_n(S)\to S^{[n]}$. For $i\in\{1,\ldots,n\}$ let 
$\tau_i\colon X_n(S)\to S$ be $\tau$ followed by the $i$-th projection. Let 
\begin{equation*}
\cF:=\bigoplus\limits_{\sigma\in\cS_n}\tau_1^{*}\cE_{\sigma(1)}\otimes\ldots\otimes
\tau_i^{*}\cE_{\sigma(i)}\otimes\ldots\otimes\tau_n^{*}\cE_{\sigma(n)}.
\end{equation*}
The pushforward $\rho_{*}(\cF)$ is torsion-free because $\cF$ is torsion-free. If all the $\cE_i$'s are locally free then  
$\rho_{*}(\cF)$ is  locally free because 
 $\rho$ is finite and flat (the latter is a highly non trivial result of Haiman, see loc.~cit.). 
The  symmetric group $\cS_n$ acts on $X_n(S)$ compatibly with its permutation action on $S^n$, and hence 
we get   an $\cS_n$-action   on  $\cF$.  Thus   we also get
 an $\cS_n$-action  on $\rho_{*}(\cF)$:
we let $\cG=\cG(\cE_1,\ldots,\cE_n)\subset \rho_{*}(\cF)$ be the sheaf of $\cS_n$-invariants.
Let
$r_i:=\rk(\cE_i)$ and $\ov{r}:=r_1\cdot\ldots\cdot r_n$.
Then we have
\begin{equation*}
\rk(\cG)= n! \ov{r},\qquad c_1(\cG) =(n-1)! \ov{r} \left[\sum\limits_{i=1}^n  \mu\left(\frac{c_1(\cE_i)}{r_i}\right)-\frac{n}{2}\delta_n\right],
\end{equation*}
where $\mu_n\colon H^2(S)\to H^2(S^{[n]})$ is the analogue of the homomorphism $\mu\colon H^2(S)\to H^2(S^{[2]})$, and $\delta_n\in H^2(S^{[n]};\ZZ)$ is the unique class such that 
$\rho^{*}\delta_n$ is the class of the exceptional divisor of $\tau$.
Next assume that  each  sheaf
 $\cE_i$ 
is simple,  that
\begin{equation}\label{jazzfever}
r_j c_1(\cE_i)=r_i c_1(\cE_j),
\end{equation}
for all $i,j$ and that
\begin{equation}\label{aussie}
\sum_{i=1}^n  \frac{v(\cE_i)^2}{r_i^2}=0.
\end{equation}
Then  
\begin{equation}\label{discingen}
\Delta(\cG)  =  \frac{(n!)^2 \ov{r}^2}{12} c_2(S^{[n]}).
\end{equation}
\end{rmk}
\subsection{Proof of Proposition~\ref{prp:extdigi}}
\setcounter{equation}{0}
Since $S$ is a $K3$ surface, the Bridgeland-King-Reid (BKR) McKay correspondence gives an equivalence between the derived categories of $\cS_2$ equivariant coherent sheaves on $S^2$ and of coherent sheaves 
on $S^{[2]}$. In particular, since $\cG=\cG(\cE_1,\cE_2)$ is the sheaf on $S^{[2]}$ corresponding to the 
$\cS_2$ equivariant sheaf $\cF=\cF(\cE_1,\cE_2)$,
we have an isomorphism 
\begin{equation}\label{extinvariante}
\Ext^p_{S^{[2]}}(\cG,\cG)\cong \Ext^p_{S^2}(\cF,\cF)^{\cS_2}.
\end{equation}
By the K\"unneth decomposition   we have 
\begin{multline}\label{bigkun}
\Ext^p_{S^2}(\cF,\cF)\cong \bigoplus\limits_{i=1}^{2}\bigoplus\limits_{a+b=p}
\Ext^{a}_{S}(\cE_{i},\cE_{i})\otimes \Ext^{b}_{S}(\cE_{3-i},\cE_{3-i})\oplus \\
\oplus 
\bigoplus\limits_{i=1}^{2}\bigoplus\limits_{a+b=p}
\Ext^{a}_{S}(\cE_{i},\cE_{3-i})\otimes \Ext^{b}_{S}(\cE_{3-i},\cE_{i}).
\end{multline}
Since $\ext^0_{S}(\cE_i,\cE_i)=1$ for  $i\in\{1,2\}$, we get that $\ext^0_{S^2}(\cF,\cF)\ge 2$ and that the invariant subspace of $\Ext^0_{S^2}(\cF,\cF)$ has dimension $1$ if and only if~\eqref{noninter} holds. This shows that Item~(1)  holds. 

A similar argument proves  Items~(2)  and~(3). 

Lastly we prove that  Item~(4) holds. For  $i\in\{1,2\}$
 let $U_i$ be a ball with center $0$ and let $\cA_i$ be a sheaf on $S\times U_i$, flat over $U_i$, such that  $\cA_i(0):=\cA_{i|S\times\{0\}}$ is isomorphic to $\cE_i$, and the associated map $(U_i,0)\to\Def(\cE_i)$ is an isomorphism. 
 Let $\cG(\cA_1,\cA_2)$ be the sheaf on $S^{[2]}\times U_1\times U_2$ that one gets by working in families, see Subsection~\ref{subsec:fasciogi}. Then  $\cG(\cA_1,\cA_2)$ is flat over 
 $U_1\times  U_2$ by Proposition~\ref{prp:gipiatto}. Thus we have a morphism of schemes
\begin{equation}
U_1\times U_2\overset{\Phi}{\lra}  \Def(\cG)
\end{equation}
which maps the  point $s=(s_1,s_2)$ to the unique $\Phi(s)\in\Def(\cG)$  such that the corresponding sheaf on $S^{[2]}$ is isomorphic to 
$\cG(\cA_1(s_1),\cA_n(s_2))$ (here $\Def(\cG)$ is a representative of the deformation space of 
$\cG$, which is a universal deformation space because  $\cG$ is simple). Let  $i\in\{1,2\}$. Since $\cE_i$ is a simple sheaf on a $K3$ surface, its deformation space is unobstructed, i.e.~$\dim U_i=\ext^1_S(\cE_i,\cE_i)$.   By Item~(2) it follows that
\begin{equation}
\dim \left(U_1\times  U_2\right)=\ext^1_{S^{[2]}}(\cG,\cG).
\end{equation}
 Thus it suffices to prove that $\Phi$ is injective. 
  By shrinking the $U_i$'s around $0$ we may assume that the following holds: if 
  $s,t\in U_1\times  U_2$ and $\alpha,\beta\in\cS_2$ are such that   
\begin{equation}\label{nessunhom}
\Hom(\cA_{\alpha(i)}(s_{\alpha(i)}),\cB_{\beta(i)}(t_{\beta(i)}))\not=0
\end{equation}
for all $i\in\{1,2\}$ (where $\cA_{\alpha(i)}(s)$ equals $\cA_{\alpha(i)|S\times\{t\}}$ and similarly for $\cB_{\beta(i)}(t))$), then $s=t$ (and $\alpha=\beta$). In fact this follows from the hypothesis that each $\cE_i$ is simple and from the vanishing 
$\ext^0_S(\cE_1,\cE_2)=\ext^0_S(\cE_2,\cE_1)=0$. By an argument similar to 
those described above, this implies that if $\Hom(\cG(\cA_1(s_1),\cA_2(s_2)),\cG(\cA_1(t_1),\cA_2(t_2))$ is non zero then $s=t$. 
This shows that $\Phi$ is injective, and concludes the proof of Item~(4). 
\qed
\subsection{Pull-back of $\cG$ to $X_2(S)$}\label{subsec:ristretto}
\setcounter{equation}{0}
The sheaf
$\rho^{*}\cG$ is obtained from $\cF$ via an elementary modification along $E$, where $E$ is the exceptional divisor of the blow up map $\tau\colon X_2(S)\to S^2$. In order to explain this we introduce some notation.
Let $D\subset S^2$ be the  diagonal of $S^2$.
 Let
\begin{equation}\label{evidcoinc}
\ov{\epsilon}\colon D  \overset{\sim}{\lra}  S 
\end{equation}
be  the isomorphism given by restriction of either one of the projections. Let  $\tau_{E}\colon E\to D$ be the restriction of $\tau$ to $E$, and let
\begin{equation}
\epsilon\coloneq \ov{\epsilon}\circ \tau_E\colon E  \lra  S. 
\end{equation}
 Let $\cR$ be the locally free sheaf on 
$E$  defined by 
\begin{equation}\label{errejeikappa}
\cR:=\epsilon^{*}(\cE_1\otimes\cE_2).
\end{equation}
One can choose an  isomorphism
\begin{equation}\label{raddoppio}
\cF_{|E}\cong \cR\oplus \cR
\end{equation}
such that the eigensheaves of the action of the involution $\sigma$ on $\cF_{|E}$ (this makes sense because 
$\sigma$ is the identity on $E$)  are given by
\begin{equation}\label{lascia}
\cF_{|E}(U)^{\pm}=\{(s,\pm s)\mid s\in \cR(U)\}.
\end{equation}
(Here  $U\subset E$ is an open subset.)
Let 
\begin{equation}\label{phijeykappa}
\begin{matrix}
\cF_{|E} & \overset{\ov{\varphi}}{\lra} & \cR \\
(a,b) & \mapsto & a-b,
\end{matrix}
\end{equation}
where the notation makes sense because of the isomorphism
 in~\eqref{raddoppio} - we assume that it has been chosen so that the equalities in~\eqref{lascia} hold. 
Let $\iota\colon E\hra X_2(S)$ be the inclusion map, and , and let $\varphi\colon \cF \to
 \iota_{*}\cR$ be the  morphism defined by the morphism  $\ov{\varphi}$ in~\eqref{phijeykappa}. Arguing as in the proof of~\cite[Proposition~5.6]{ogfascimod} one gets that
 $\rho^{*}\cG$   fits into the exact sequence
\begin{equation}\label{gimodel}
0\lra \rho^{*}\cG \lra \cF\overset{\varphi}{\lra} 
 \iota_{*}\cR\lra 0.
\end{equation}
\subsection{Proof of Proposition~\ref{prp:gimodulare}} 
\setcounter{equation}{0}
The rank of $\cG$ can be computed away from  the branch locus of $\rho$, and it is equal to the rank of 
$\cF$, i.e.~$2 r_1 r_2$.  
By the exact sequence in~\eqref{gimodel} we can express $\ch(\rho^{*}\cG)$
   via  $\ch(\cF)$  and the Chern character of the sheaf $\iota_{*}\cR$.
 Applying the GRR Theorem we get that \emph{modulo} $H^6(X_2(S),\QQ)$ we have
\begin{equation}\label{tensoriem}
\ch(\iota_{*}\cR)=
r_1 r_2 e+\frac{1}{2} c_1(\cF)\cdot e
-\frac{r_1 r_2}{2} e^2
\end{equation}
Hence we get that
\begin{equation}
c_1(\rho^{*}\cG)=c_1(\cF)-c_1(\iota_{*}(\cR))=\left(\sum\limits_{1\le a,b\le 2}^2 \frac{r_1 r_2}{r_a} \tau_b^{*} c_1(\cE_a)\right)-r_1 r_2 e.
\end{equation}
By the equalities in~\eqref{tiromu} we get that 
\begin{equation}\label{beltrami}
\rho^{*}c_1(\cG)=
\rho^{*}( r_2\mu(c_1(\cE_1))+r_1\mu(c_1(\cE_2))-r_1 r_2\delta).
\end{equation}
 The pull-back homomorphism $\rho^{*}\colon H^2(S^{[n]})\to 
H^2(X_n(S))$ is injective because $\rho$ is a finite map. Hence the second equality in~\eqref{rangoeciuno}
 follows from the equality in~\eqref{beltrami}. 
 This finishes the proof of Item~(1) of Proposition~\ref{prp:gimodulare}. 
 
 Next we prove Item~(2). 
 It suffices to show    that 
\begin{equation}\label{ennio}
\rho^{*}\Delta(\cG)=\frac{ r_1^2 r_2^2}{3}\rho^{*} c_2(S^{[2]}). 
\end{equation}
By the equality in~\eqref{tensoriem} we have
\begin{multline}\label{formulona}
\rho^{*}\ch_2(\cG)=
\ch_2(\cF)-\ch_2\left(\iota_{*}(\cR)\right)=\\
=\sum\limits\limits_{1\le a,b\le 2} \left(\frac{r_1 r_2}{r_a}\tau_b^{*} \ch_2(\cE_a)+
\frac{r_1 r_2}{r_a r_b}\tau_1^{*} \ch_1(\cE_a)\tau_2^{*} \ch_1(\cE_b)\right)\restrict{X_2(S)}-\\
-\frac{1}{2} e\cdot\left( \sum\limits\limits_{1\le a,b\le 2} \frac{r_1 r_2}{r_a}\tau_b^{*} c_1(\cE_a)\right)
+\frac{r_1 r_2}{2} e^2.
\end{multline}
We recall that 
\begin{equation}\label{gelato}
\rho^{*}\Delta(\cG)=-4\ov{r}\rho^{*}\ch_2(\cG)+\rho^{*}\ch_1(\cG)^2.
\end{equation}
Let 
\begin{equation}
\lambda=\frac{c_1(\cE_1)}{r_1}=\frac{c_1(\cE_2)}{r_2}.
\end{equation}
(Recall the equality~\eqref{sameslope}.)  The equalities in~\eqref{formulona} and~\eqref{gelato}, together with a few computations, give that
\begin{equation}\label{difretta}
\rho^{*}\Delta(\cG)=
-4 r_1^2 r_2^2  \sum\limits_{1\le a,b\le 2} \tau_b^{*} \left(\frac{\ch_2(\cE_a)}{r_a}\right)+4 r_1^2 r_2^2\left(\sum\limits_{a=1}^2 \tau_a^{*}\lambda^2\right)-\frac{r_1^2 r_2^2}{2}e^2. 
\end{equation}
Let $\eta\in H^4(S)$ be the orientation class. By definition of Mukai pairing we have
\begin{equation}\label{utileq}
\frac{\ch_2(\cE_a)}{r_a}=-\frac{v(\cE_a)^2}{2 r^2_a}\eta-\eta+\frac{\lambda^2}{2}.
\end{equation}
Replacing in the right hand side of~\eqref{difretta} the above expression for $\ch_2(\cE_a)/r_a$, and recalling the equality in~\eqref{sommaquad}, we get that 
\begin{equation}\label{bambola}
\rho^{*}\Delta(\cG)=8 r_1^2 r_2^2  (\tau_1^{*}\eta+\tau_2^{*}\eta)
-\frac{r_1^2 r_2^2 }{2}e^2. 
\end{equation}
On the other hand we have (see~\cite[Proposition~2.1]{og-rigidi-su-k3n}) 
\begin{equation}\label{levicivita}
\rho^{*}c_2(S^{[2]}) = 24(\tau_1^{*}(\eta)+\tau_2^{*}(\eta))-3 e^2.
\end{equation}
The validity of the equality in~\eqref{ennio} follows from the equalities in~\eqref{bambola} and 
 in~\eqref{levicivita}.
\qed
\subsection{Modular sheaves}\label{subsec:ancoramod}
\setcounter{equation}{0}
Let $X$ be  a HK manifold  of dimension $2n$.  The \emph{discriminant} of a  torsion-free sheaf $\cF$ on $X$ is
\begin{equation}\label{eccodisc}
 \Delta(\cF) :=2\rk(\cF) c_2(\cF)-(\rk(\cF)-1) c_1(\cF)^2=-2\rk(\cF)\ch_2(\cF)+\ch_1(\cF)^2.
\end{equation}
 The sheaf $\cF$  is modular (see~\cite{ogfascimod}) if 
 there exists $d(\cF)\in\QQ$ such that  
\begin{equation}\label{fernand}
\int_X \Delta(\cF)  \alpha^{2n-2}=d(\cF) (2n-3)!! q_X(\alpha)^{n-1}
\end{equation}
for all $\alpha\in H^2(X)$, where $q_X$ is the Beauville-Bogomolov-Fujiki quadratic form of $X$. 
\begin{expl}\label{expl:didik3}
If  $X$ is a $K3$ surface, then every torsion-free sheaf $\cF$ on $X$ is modular and $d(\cF)=\int_X\Delta(\cF)=v(\cF)^2+2\rk(\cF)^2$.  
\end{expl}
\begin{expl}\label{expl:didigi}
Let $\cG=\cG(\cE_1,\ldots,\cE_n)$ be the sheaf  on $S^{[n]}$ appearing in Proposition~\ref{prp:gimodulare} for $n=2$ and in Remark~\ref{rmk:cosgen} for general $n$. Then $\cG$ is modular. In fact this follows from the equality in~\eqref{discform} and the formula $\int_{S^{[n]}}c_2(S^{[n]})\alpha^{2n-2}=6(n+3)(2n-3)!! q_{S^{[n]}}(\alpha)^{n-1}$.
Thus $d(\cG)=(n+3)(n!)^2 \ov{r}^2/2$. 
\end{expl}
We recall that the Fujiki constant of $X$ (sometimes called the small Fujiki constant) is characterized by the validity of the equality
\begin{equation}
\int_X\alpha^{2n}=(2n-1)!! c_X q_X(\alpha)^{n}
\end{equation}
for  all $\alpha\in H^2(X)$.
\begin{dfn}\label{dfn:adieffe}
Let $X$ be a HK manifold, and let $\cF$ be a modular torsion free sheaf  on $X$.
Then
\begin{equation}\label{esmeralda}
{\mathsf a}(\cF):=\frac{\rk(\cF)^2 \cdot d(\cF) }{4c_X},
\end{equation} 
where $d(\cF)$ is as in Definition~\ref{fernand} and $c_X$ is the Fujiki constant of $X$.
\end{dfn}
\begin{expl}\label{expl:adigi}
Let $\cG(\cE_1,\ldots,\cE_n)$ be as in Example~\ref{expl:didigi}. Then 
\begin{equation}\label{adimod}
{\mathsf a}(\cG(\cE_1,\ldots,\cE_n))=(n+3)(n!)^4 \ov{r}^4/8. 
\end{equation}
In fact this equality follows from the formula for $d(\cG)$ given in Example~\ref{expl:didigi} and the equality $c_{S^{[n]}}=1$.
\end{expl}
\begin{dfn} 
Let $S$ be a $K3$ surface, and let $v=(r,l,s)$ be a Mukai vector on $S$. We set ${\mathsf a}(v)\coloneq (r^2(v^2+2r^2)/4$.  
\end{dfn} 
\begin{dfn}\label{dfn:marina}
Let $X$ be a HK manifold of dimension $2n>2$. Let $w=(r,l,s)\in \NN_{+}\times \NS(X)\times H^{2,2}_{\ZZ}(X)$, and assume that  
 there exists $d\in\QQ$ such that  
\begin{equation}\label{coraline}
\int_X s \cdot \alpha^{2n-2}=d (2n-3)!! q_X(\alpha)^{n-1}
\end{equation}
for all $\alpha\in H^2(X)$. 
We set ${\mathsf a}(w)\coloneq r^2 d/4 c_X$.
\end{dfn} 
\begin{rmk}\label{rmk:numgenhk}
Let $v=(r,l,s)$ be a Mukai vector on a $K3$ surface $S$. If $\cF$ is a sheaf on $S$ such that $v(\cF)=v$, then  ${\mathsf a}(\cF)={\mathsf a}(v)$ (see Example~\ref{expl:didik3}).
Let $X$ be a HK manifold of dimension $2n>2$, and  let $w=(r,l,s)$ be as in Definition~\ref{dfn:marina}. 
If $\cF$ is a sheaf on $X$ such that $w(\cF)=w$, then  ${\mathsf a}(\cF)={\mathsf a}(w)$. 
\end{rmk}

\section{Examples}\label{sec:eccoesempi}
\subsection{Preliminaries} 
\setcounter{equation}{0}
Let $S$ be a $K3$ surface, and let $\cE_1,\cE_2$ be  sheaves on $S$. For $i\in\{1,2\}$ let
\begin{equation}
v(\cE_i)=(r_i,l_i,s_i)
\end{equation}
be the Mukai vector of $\cE_i$.
\begin{lmm}\label{lmm:alterpereconi}
If the hypotheses of Proposition~\ref{prp:gimodulare} (including~\eqref{sameslope} and~\eqref{sommaquad})  hold  
then 
\begin{enumerate}
\item
$v(\cE_i)^2=0$ for $i\in\{1,2\}$ and $v(\cE_1)$, $v(\cE_2)$ are proportional, or
\item
up to reindexing we have $v(\cE_1)^2=-2$ and $v(\cE_2)^2> 0$, and 
 there exists   $a\in\NN_{+}$ such that $r_2=a r_1$ and $l_2=a l_1$.
\end{enumerate}
\end{lmm}
\begin{proof}
Suppose that  $v(\cE_1)^2=v(\cE_2)^2=0$. Then by  the equality $r_2 l_1=r_1 l_2$ (see Equation~\eqref{sameslope}) we may write $v(\cE_1)=t v(\cE_2)+(0,0,s)$ for some $t,s\in\QQ$. Note that  $t\not=0$. Since $v(\cE_1)^2=v(\cE_2)^2=0$ it follows that $s=0$, and hence $v(\cE_1)$, $v(\cE_2)$ are proportional. 

Suppose that  $v(\cE_1)^2$, $v(\cE_2)^2$ are not both zero.
Since $r_2^2 v(\cE_1)^2+r_1^2 v(\cE_2)^2=0$ there exists  $i\in\{1,2\}$ 
such that $v(\cE_i)^2<0$. Reindexing we may assume that $i=1$.  By simplicity of $\cE_1$ we get that $v(\cE_1)^2=-2$, i.e.~that 
\begin{equation}\label{radice}
r_1 s_1-l_1^2/2=1.
\end{equation}
 Hence $\divisore(l_1)$ and $r_1$ are coprime. 
 The relation
$r_2 l_1=r_1 l_2$
gives that there exists $a\in\NN_{+}$  such that  $r_2=a r_1$, $l_2=a l_1$ (because $\divisore(l_1)$ and $r_1$ are coprime). 
\end{proof}
\begin{rmk}
Let $\cE_1,\ldots,\cE_n$ be sheaves on a $K3$ surface $S$ as in Remark~\ref{rmk:cosgen}, and suppose that the equalities in~\eqref{jazzfever} and~\eqref{aussie} hold. For $j\in\{1,\ldots,n\}$ let $v(\cE_j)=(r_j,l_j,s_j)$. 
Then either 
$v(\cE_j)^2=0$ for all $j\in\{1,\ldots,n\}$ and the Mukai vectors $v(\cE_1),\ldots,v(\cE_n)$ are proportional 
or else, up to reindexing,
$v(\cE_1)^2=-2$ and for all $j>1$ we have $v(\cE_j)^2\ge 0$  and $r_j=a_j r_1$, $l_j=a_j l_1$ where $a_j\in\NN_{+}$.
\end{rmk}
\begin{rmk}
The modular sheaves given by $\cG(\cE_1,\ldots,\cE_n)$ where the Mukai vectors $v(\cE_1),\ldots,v(\cE_n)$ are isotropic (primitive) and  all equal are studied by Markman in~\cite{markman-1-obstructed} (see Section~11). 
\end{rmk}
\subsection{A series of examples}\label{subsec:esempiessedue}
\setcounter{equation}{0}
We discuss examples of $\cE_1,\cE_2$ such  that Item~(2) of 
Lemma~\ref{lmm:alterpereconi} holds. 
\begin{lmm}\label{lmm:vuquad}
Let $(S,h)$ be a polarized $K3$ surface. Let 
$\cE_1,\cE_2$ be Gieseker-Maruyama stable torsion free sheaves on $S$, with $\cE_1$ spherical (and hence locally free). Assume also that there exists  $a\in\NN_{+}$ such that
\begin{equation}\label{vuunovudue}
v(\cE_2)=a v(\cE_1)-\frac{2a}{r_1}\left(0, 0,1\right).
\end{equation}
Then $\cG=\cG(\cE_1,\cE_2)$ is a torsion free simple sheaf on $S^{[2]}$, and
\begin{equation}\label{sanmichele}
w(\cG) = ar_1 \left(2 r_1, 2  \mu(l_1)-r_1\delta,   \frac{ a r_1^3}{3} c_2(S^{[2]}) \right).
\end{equation}
If in addition 
\begin{equation}\label{boatto}
\Hom(\cE_2,\cE_1)=0,
\end{equation}
 then 
\begin{equation}\label{zerounodue}
 \ext^1_{S^{[2]}}(\cG,\cG)=2a^2+2,\quad 
\ext^2_{S^{[2]}}(\cG,\cG)=2,
\end{equation}
 $\cG$ has unobstructed deformations, and $\Def(\cG)$ is identified with 
$\Def(\cE_1)\times\Def(\cE_2)$ via the map in~\eqref{defoprod}.
\end{lmm}
\begin{proof}
 We claim that 
 $\Hom(\cE_1,\cE_2)=0$.
 In fact by the equality in~\eqref{vuunovudue} one gets that
 $\chi(S,\cE_1(n))/r(\cE_1)\succ\chi(S,\cE_2(n))/r(\cE_2)$ (here $\succ$ means that the left hand  is greater than the right hand  for $n\gg 0$, in fact for all $n$ in this specific case), and hence $\Hom(\cE_1,\cE_2)=0$ by stability. By Proposition~\ref{prp:extdigi} it follows that $\cG$ is simple. 
 The equalities in~\eqref{sameslope} and~\eqref{sommaquad} hold because of  the equality in~\eqref{vuunovudue}, and thus the equalities in~\eqref{sanmichele}
hold by Proposition~\ref{prp:gimodulare}. 

The validity of the remaining statements assuming the vanishing in~\eqref{boatto} follows from Proposition~\ref{prp:extdigi}, because 
 $\Hom(\cE_1,\cE_2)=0$. 
\end{proof}
\begin{rmk}\label{rmk:casoprim}
Let  $\cE_1,\cE_2$ be as in Lemma~\ref{lmm:alterpereconi}, and let $\cG\coloneq \cG(\cE_1,\cE_2)$.
If Item~(1) of Lemma~\ref{lmm:alterpereconi} holds, with $v(\cE_1)=v(\cE_2)$ and $\cE_1,\cE_2$ stable non isomorphic vector bundles,   then~\eqref{noodles} gives that $\ext^2_{S^{[2]}}(\cG,\cG)  = 6$. 
If Item~(2) of Lemma~\ref{lmm:alterpereconi} holds, with  $\cE_1,\cE_2$ slope stable,  then~\eqref{noodles} gives that 
$\ext^2_{S^{[2]}}(\cG,\cG)  = 2$.  There are examples with  $\cE_2$ Gieseker-Maruyama  stable but not slope stable with 
$\ext^1_{S^{[2]}}(\cE_1,\cE_2)$ arbitrarily large, and hence $\ext^2_{S^{[2]}}(\cG,\cG)$ arbitrarily large by~\eqref{noodles}. However in these cases $\cG$ is  unstable. This motivates our expectation that $\cG$ belongs to a connected smooth component of the corresponding moduli space of (semi)stable sheaves on $S^{[2]}$. The main result in~\cite{bottini-og10-reloaded} gives further evidence towards the expectation that moduli spaces of modular sheaves (or at least of projectivelt hyperholomorphic sheaves) often smooth.
\end{rmk}
\begin{rmk}\label{rmk:casoprim}
Set $r_1=2a$ or $r_1=a$ with $a$ odd in Lemma~\ref{lmm:vuquad}. In other words
let
\begin{enumerate}
\item
$v(\cE_1)=( 2a,  l_1, s_1 )$ and $v(\cE_2)=( 2a^2, a l_1,a s_1 -1)$, or
\item
$v(\cE_1)=( a,  l_1, s_1 )$ and $v(\cE_2)=( a^2, a l_1,a s_1 -2)$ and $a$ is odd.
\end{enumerate}
Then the  vector $v(\cE_2)$ is primitive. Conversely, if in Lemma~\ref{lmm:vuquad} the vector $v(\cE_2)$ is primitive, then  either Item~(1) or Item~(2) holds. 
\end{rmk}
\subsection{Fatighenti's example}\label{subsec:fatighenti}
\setcounter{equation}{0}
Let $Y\subset\PP^5$ be a smooth cubic hypersurface. Let $X\subset\GR(1,\PP^5)$ be the variety parametrizing lines in $Y$, and let 
$h$ be the Pl\"ucker polarization of $X$. Then $(X,h)$ is a general  HK of type $K3^{[2]}$ with polarization of square $6$ and divisibility $2$.
Let $\cQ$ be the restriction to $X$ of the tautological quotient rank $4$ vector bundle on $\GR(1,\PP^5)$. 
Then  $\cQ$ is a rigid modular vector bundle which is stable if $Y$ is general, and belongs to the class of vector bundles studied
 in~\cite{ogfascimod,og-rigidi-su-k3n}. In~\cite{fatighenti-esempi} it is shown that 
$\bigwedge^2\cQ$ is stable, and that 
\begin{equation}\label{accaunodue}
h^1(X,End(\bigwedge\nolimits^2\cQ))=20,\quad h^2(X,End(\bigwedge\nolimits^2\cQ))=2.
\end{equation}
 Let
\begin{equation*}
w\coloneq 3(2,h,c_2(X))=\left(\rk\left(\bigwedge\nolimits^2\cQ\right),c_1\left(\bigwedge\nolimits^2\cQ\right),
\Delta\left(\bigwedge\nolimits^2\cQ\right)\right).
\end{equation*}
A computation gives that $[\bigwedge\nolimits^2\cQ]\in M_w(X,h)$. 
Here we show that $\bigwedge\nolimits^2\cQ$ is a deformation of $\cG(\cE_1,\cE_2)$ for suitable $\cE_1,\cE_2$. More precisely, let 
$(S,D)$ be a polarized $K3$ surface with $D\cdot D\equiv 2\pmod{4}$, and let $\cF$ be the  stable spherical vector bundle on $S$ with Mukai vector
\begin{equation}
v(\cF)=(2,D,(D\cdot D+2)/4).
\end{equation}
(Abusing notation we denote by the same symbol $D$ and its  Poincar\'e dual.) A straightforward computation gives that
\begin{equation}
v(\Sym\nolimits^2\cF)=\left(3,3D,\frac{3D\cdot D}{2}-3\right)=3v\left(\bigwedge\nolimits^2\cF\right)-(0,0,6).
\end{equation}
Hence the equality in~\eqref{vuunovudue} is satisfied  (with $r_1=1$ and $a=3$) by
\begin{equation} 
\cE_1\coloneq\bigwedge\nolimits^2\cF,\qquad \cE_2\coloneq\Sym^2\cF.
\end{equation}
 Here we prove that the remaining hypotheses of Lemma~\ref{lmm:vuquad} (i.e.~stability of $\Sym^2\cF$ and the validity 
 of~\eqref{boatto}) hold  under additional hypotheses on $(S,D)$.
The result below follows from surjectivity of the period map for $K3$ surfaces.
\begin{clm}\label{clm:neronsevero}
Let $m_0,d_0$ be positive natural numbers. There exist $K3$ surfaces $S$  with an elliptic fibration 
$\varepsilon\colon S\to \PP^1$  and  elliptic fiber $C\subset S$ 
such that 
\begin{equation}\label{neronsevero}
\NS(S)=\ZZ[D]\oplus\ZZ[C], \quad
D\cdot D=2m_0,\quad D\cdot C=d_0.
\end{equation}
\end{clm}
\begin{lmm}\label{lmm:symok}
Let $S$ be an elliptic $K3$ surface as in Claim~\ref{clm:neronsevero}, and suppose that $d_0$ is odd and $d_0>3(2m_0+1)$.
 Then 
\begin{enumerate}
\item
 $\Sym^2\cF$ is a slope stable vector bundle, and
\item
there is no non zero map $\Sym^2\cF\lra \bigwedge\nolimits^2\cF$
\end{enumerate}
\end{lmm}
\begin{proof} 
By Proposition~6.2 in~\cite{ogfascimod} the vector bundle $\cF$ is slope stable for any polarization, in particular for $D$. 

\n
(1): Let $t\in\PP^1$ and let
$C_t\coloneq\varepsilon^{-1}(t)$ be the corresponding elliptic fiber. The restriction $\cF_t \coloneq \cF_{|C_t}$ is (slope) stable by 
Proposition~6.2 loc.cit. If $C_t$ is smooth it follows that 
\begin{equation}\label{threemice}
\Sym\nolimits^2\cF_t\cong \bigoplus_{0\not=\alpha\in C_t[2]}\cO_{C_t}(D+\alpha),
\end{equation}
where $C_t[2]<\Pic^0(C_t)$ is the $2$-torsion subgroup. 
In fact this can be proved as follows. By general results $\cF_t\otimes\cF_t\cong L_0\oplus L_1\oplus L_2\oplus L_3$ where each $L_i$ is a line bundle. Of course we may assume that $L_0=\bigwedge^2\cF_t\cong\cO_{C_t}(D)$.  
Let $\alpha\in C_t[2]$ and let $\tau_{\alpha}\colon C_t\to C_t$ be translation by $\alpha$. Then 
$\tau_{\alpha}^{*}(\cF_t\otimes\cF_t)\cong \cF_t\otimes\cF_t$. Hence the translation action of $C_t[2]$  on $\Pic(C_t)$  permutes
the isomorphism classes $[L_0], [L_1], [L_2], [L_3]$. This forces  
\begin{equation*}
\{[L_1], [L_2], [L_3]\}=\{[\cO_{C_t}(D+\alpha)], 
[\cO_{C_t}(D+\beta)], [\cO_{C_t}(D+\alpha+\beta)]\},
\end{equation*}
where $\alpha,\beta\in C_t[2]$ are non zero and distinct.
We have proved the validity of~\eqref{threemice}. 

Suppose that 
$\Sym\nolimits^2\cF$  is not slope stable. 
Since it  is slope polystable by general results (see for example~\cite[Theorem~3.2.11]{huy-lehn-book}) and it has rank $3$, it follows that 
\begin{equation}\label{nicowill}
\Sym\nolimits^2\cF\cong \cO_S(D)\oplus\cV,
\end{equation}
where  $\cV$ is a rank $2$ vector bundle. 
The above decomposition is incompatible with the direct sum decomposition in~\eqref{threemice}  because there is no nonzero map
\begin{equation}
\cO_S(D)_{|C_t}=\cO_{C_t}(D)\lra \bigoplus_{0\not=\alpha\in C_t[2]}\cO_{C_t}(D+\alpha).
\end{equation}
It follows that $\Sym\nolimits^2\cF$  is slope stable. 

\n
(2): Restricting a map  $\varphi\colon\Sym^2\cF\lra \bigwedge\nolimits^2\cF$ to a smooth fiber $C_t$ and recalling the decomposition 
in~\eqref{threemice} we get that the restriction of $\varphi$ to $C_t$ is zero. Since $\Sym^2\cF$ is locally free it follows that 
$\varphi=0$.
\end{proof}
Let hypotheses be as in Lemma~\ref{lmm:symok}. 
 By Lemma~\ref{lmm:vuquad} the vector bundle 
\begin{equation}
\cG_0\coloneq\cG(\bigwedge^2\cF,\Sym^2\cF)
\end{equation}
 is modular,  one has
\begin{equation*}
(\rk(\cG_0),c_1(\cG_0),\Delta(\cG_0))=3(2,(2\mu(D)-\delta),3c_2(S^{[2]}),
\end{equation*}
 and  $\cG_0$ has unobstructed deformations given by $\cG(\bigwedge^2\cF,\cA)$, where $\cA$ is a (nearby) deformation of 
 $\Sym^2\cF$. Now notice that if $D\cdot D=2$ then $2\mu(D)-\delta$ has square $6$ and divisibility $2$. It follows that if $(X,h)$ is a general deformation of $(S^{[2]},2\mu(D)-\delta)$ then $(X,h)$ is isomorphic to the variety of lines on a general cubic hypersurface in  $\PP^5$ polarized by the Pl\"ucker line bundle (note: $2\mu(D)-\delta)$ is not ample). 
\begin{prp}
Let hypotheses be as in Lemma~\ref{lmm:symok}. 
Assume in addition that $D\cdot D=2$, and let $(X,h)$ be a general deformation of $(S^{[2]},2\mu(D)-\delta)$. Then the 
couple $(S^{[2]},\cG(\bigwedge^2\cF,\Sym^2\cF))$  deforms to the couple $(X,\bigwedge^2\cQ)$ where  
$\cQ$ is the restriction to $X$ of the tautological quotient rank $4$ vector bundle on $\GR(1,\PP^5)$. 
\end{prp}
\begin{proof}
Let $\cF[2]^{+}$ be the (modular) vector bundle on $S^{[2]}$ associated to $\cF$ according to Definition~5.1 in~\cite{ogfascimod}, see Subsection~\ref{subsec:fasciogi}. Then 
\begin{equation}
(\rk(\cF[2]^{+}),c_1(\cF[2]^{+}),\Delta(\cF[2]^{+}))=(4,2\mu(D)-\delta,c_2(S^{[2]})).
\end{equation}
(See Proposition~5.2 in loc.~cit.) Moreover the couple $(S^{[2]},\cF[2]^{+})$ deforms to $(X,\cQ)$ where $(X,h)$ is a general deformation of $(S^{[2]},2\mu(D)-\delta)$. Thus it suffices to prove that there is an isomorphism
\begin{equation}\label{trombascala}
\bigwedge^2\cF[2]^{+}\cong \cG\left(\bigwedge^2\cF,\Sym^2\cF\right).
\end{equation}
Let notation be as in Subsection~\ref{subsec:ristretto}. We have the exact sequence 
\begin{equation}\label{sollevo}
0  \lra  \rho^{*}\cF[2]^{+}  \lra   \tau_1^{*}(\cF)\otimes \tau_2^{*}(\cF)   \lra \iota_{*}\left(\epsilon^{*}\bigwedge^2 \cF\right)  \lra   0,
\end{equation}
see Equation~(5.2.2) in~\cite{ogfascimod}. Taking the second exterior product we get an exact sequence described as follows. If $V,W$ are (complex) vector spaces we may define an isomorphism
\begin{equation*}
f\colon \Sym^2 V\otimes\bigwedge^2 W\oplus \bigwedge^2 V\otimes \Sym^2 W  \overset{\sim}{\lra}   \bigwedge^2 (V\otimes W) 
\end{equation*}
 by letting
\begin{equation*}
f(v_1 v_2\otimes w_1\wedge w_2, 0) \coloneq v_1\otimes w_1\wedge v_2\otimes w_2-
v_1\otimes w_2\wedge v_2\otimes w_1,
\end{equation*}
and similarly for $f(0,v'_1 v'_2\otimes w'_1\wedge w'_2)$.

It follows that by taking the second exterior product of the terms in the Exact Sequence~\eqref{sollevo} we get the exact sequence
\begin{multline}\label{esternodue}
0  \lra  \rho^{*}\left(\bigwedge^2\cF[2]^{+}\right)  \lra  \tau_1^{*}(\Sym^2\cF)\otimes \tau_2^{*}\left(\bigwedge^2\cF\right)\oplus  
\tau_1^{*}\left(\bigwedge^2\cF\right)\otimes \tau_2^{*}(\Sym^2\cF)
  \overset{\psi}{\lra} \\
 \overset{\psi}{\lra} \iota_{*}\left(\epsilon^{*}\Sym^2\cF\otimes\bigwedge^2 \cF\right)  \lra   0.
\end{multline}
Now compare the above exact sequence with the one in~\eqref{gimodel} for $\cG=\cG(\bigwedge^2\cF,\Sym^2\cF)$. The middle terms are equal, and the quotient map $\psi$ in~\eqref{esternodue} is identified with the quotient map $\varphi$ in~\eqref{gimodel}. Hence we get an isomorphism between 
$ \rho^{*}\left(\bigwedge^2\cF[2]^{+}\right)$ and  $\rho^{*}\cG\left(\bigwedge^2\cF,\Sym^2\cF\right)$. The isomorphism descends to an isomorphism~\eqref{trombascala}. 

\end{proof}

\section{Atomicity/Non atomicity}\label{sec:nonatomico}
\subsection{The main result} 
\setcounter{equation}{0}
Let $\cE_1,\cE_2$ be sheaves on a $K3$ surface $S$ satisfying the hypotheses of Proposition~\ref{prp:gimodulare}. Hence the sheaf $\cG=\cG(\cE_1,\cE_2)$ is modular. 
The main result of the present section is the following. 
\begin{prp}\label{prp:atomicose}
The sheaf $\cG=\cG(\cE_1,\cE_2)$ is atomic (see Subsection~\ref{subsec:richiami}) if and only if
\begin{equation}\label{dueisotropi}
v(\cE_1)^2=v(\cE_2)^2=0. 
\end{equation}
\end{prp}
\begin{rmk}
Markman showed in~\cite{{markman-1-obstructed}}   that $\cG(\cE_1,\cE_2)$ is  numerically $1$-obstructed (i.e.~atomic) if the equalities in~\eqref{dueisotropi} hold  (or more generally that $\cG(\cE_1,\ldots,\cE_n)$ is atomic if $v(\cE_1)^2=\ldots=v(\cE_n)^2=0$). The point of our computation is to show the reverse implication: if  $\cG(\cE_1,\cE_2)$ is  atomic then the equalities in~\eqref{dueisotropi} hold.
\end{rmk}
\subsection{Recap of work by Taelman, Markman and Beckmann}\label{subsec:richiami}
\setcounter{equation}{0}
 In the present subsection we recall the notion, introduced by Beckmann, of \lq\lq extended Mukai vector\rq\rq\ of  a sheaf on a HK manifold $X$ of dimension $2n$. References are~\cite{taelman-der-hks,markman-1-obstructed,beckmann-ext-muk,beckmann-atomic}.
 
 Let $ H^2(X):=H^2(X;\QQ)$. The extended rational Mukai lattice of $X$ is given by the rational vector space
\begin{equation}\label{reticoloesteso}
\wt{H}(X):=\QQ\alpha\oplus H^2(X)\oplus \QQ\beta
\end{equation}
with the bilinear symmetric form  $\wt{b}$ defined as follows. The direct sum decomposition in~\eqref{reticoloesteso} is orthogonal for $\wt{b}$, the restriction to $H^2(X)$ equals the BBF bilinear symmetric form, and 
\begin{equation}
\wt{b}(\alpha,\alpha)=\wt{b}(\beta,\beta)=0,\quad \wt{b}(\alpha,\beta)=-1.
\end{equation}
For $v\in \wt{H}(X)$ we let $\wt{q}(v)=\wt{b}(v,v)$. Let ${\mathfrak g}_{\QQ}(X)$ be the rational Looijenga-Lunts-Verbitsky algebra of $X$, and let $\SH(X)\subset H(X)$ be the Verbitsky subalgebra, generated over $\QQ$ by $H^2(X)$. One has an isomorphism of Lie algebras $ {\mathfrak g}_{\QQ}(X)\cong {\mathfrak s}{\mathfrak o}(\wt{H}(X))$ such that there is an embedding of 
${\mathfrak s}{\mathfrak o}(\wt{H}(X))$-modules
\begin{equation}
\SH(X)\overset{\Psi}{\lra}\Sym^n \wt{H}(X)
\end{equation}
described as follows. Associate to $\lambda\in H^2(X)$ the element $e_{\lambda}\in {\mathfrak s}{\mathfrak o}(\wt{H}(X))$ defined by
\begin{equation}
e_{\lambda}(\alpha)=\lambda,\qquad e_{\lambda}(\mu)=\wt{b}(\lambda,\mu)\beta\ \ \forall \mu\in H^2(X),\qquad e_{\lambda}(\beta)=0.
\end{equation}
Note that $e_{\lambda}$ and $e_{\mu}$ commute for any  $\lambda,\mu\in H^2(X)$. One defines $\Psi$ by letting
\begin{equation}
\Psi(\lambda_1\ldots\lambda_k):=e_{\lambda_1}\ldots e_{\lambda_k}\left(\frac{\alpha^n}{n!}\right)
\end{equation}
The map $\Psi$ is an isometric embedding with respect to the  non degenerate bilinear forms on $\SH(X)$  and $\Sym^n \wt{H}(X)$ defined as follows. 
The Mukai pairing $(,)_{\rm M}$ on  $\SH(X)$ is defined by requiring that 
\begin{equation}
(\xi,\eta)_{\rm M}:=(-1)^p\int_X \xi\cdot\eta  
\end{equation}
if $\xi\in \SH^{2p}(X),\eta\in \SH^{4n-2p}(X)$ (we let $\SH^{2d}(X):= \SH(X)\cap H^{2d}(X)$),
and that 
\begin{equation}
\SH^{2p}(X)\bot \SH^{2q}(X)\ \ \text{if $p+q\not=2n$.}
\end{equation}
Note: if $X$ is a $K3$ surface, then $(,)_{\rm M}$ is the opposite of the classical Mukai pairing. The bilinear form $\wt{b}_{[n]}$ on  $\Sym^n \wt{H}(X)$ is defined (following Beckmann, see~[p.119]\cite{beckmann-ext-muk})
\begin{equation}
\wt{b}_{[n]}(x_1\ldots x_n,y_1\ldots y_n):=c_X(-1)^n\sum_{\sigma\in\cS_n}\prod_{i=1}^n \wt{b}(x_i,y_{\sigma(i)}).
\end{equation}
As stated above, one has
\begin{equation}
\wt{b}_{[n]}(\Psi(\xi),\Psi(\eta))=(\xi,\eta)_{\rm M}\quad \forall \xi,\eta\in\SH(X).
\end{equation}
Since  $(\xi,\eta)_{\rm M}$ and $\wt{b}_{[n]}$ are non degenerate, the orthogonal projection
\begin{equation}
 \Sym^n \wt{H}(X)\overset{T}{\lra} \SH(X)
\end{equation}
is well defined.  

There is also a well defined orthogonal projection
\begin{equation}
\begin{matrix}
H(X) & \lra & \SH(X) \\
\eta & \mapsto & \ov{\eta}
\end{matrix}
\end{equation}
because the intersection form on $H(X)$ and its restriction  to $\SH(X)$ are non degenerate. 
 Let $\cF$ be a sheaf on $X$ (or an object of $D^{b}(X)$). The associated Mukai vector is given by
\begin{equation}
v(\cF):=\ch(\cF)\sqrt{\Td(X)}.
\end{equation}
Note that $v(\cF)\in H(X)$, i.e.~it's a  rational cohomology class. Note that $\ov{v(\cF)}\in \SH(X)$. 
\begin{dfn}[Definition 4.15 in~\cite{beckmann-ext-muk}]
Let $v\in \wt{X}$. Then $v$ is 
 an extended Mukai vector of $\cF$ if 
\begin{equation}
\Span(\ov{v(\cF)})=\Span\{T(v^n)\}.
\end{equation}
\end{dfn}
Abusing notation, if an extended Mukai vector of $\cF$ exists  we denote it by $\wt{v}(\cF)$. 
\begin{rmk}
If $\cF$ is atomic (see~\cite{beckmann-atomic}), i.e.~numerically $1$-obstructed (see~\cite{markman-1-obstructed}), then it has an extended Mukai vector. For HK manifolds of type $K3^{[2]}$ the two notions are equivalent, but in general there is no reason why 
having an extended Mukai vector should imply atomic.
\end{rmk}
\subsection{An  \lq\lq extended Mukai vector\rq\rq\ of $\cF$ is determined by $r(\cF)$, $c_1(\cF)$ and $\Delta(\cF)$} 
\setcounter{equation}{0}
Let $\cF$ be a sheaf with positive rank $r$. Suppose that an  extended Mukai vector of $\cF$. Markman and Beckmann showed that one may assume that $\wt{v}(\cF)=r\alpha+c_1(\cF)+s\beta$. They also showed that $\cF$ is modular. 
Bottini, see~\cite[Corollary~3.10]{bottini-og10}, showed how to determine $s$ from $r$, $c_1(\cF)$ and $\Delta(\cF)$ if $\dim X=4$. 

In the present subsection we extend Bottini's computation to arbitrary $X$. First we recall a few definitions. Following Beckmann we let 
${\mathsf q}_{2i}\in\SH^{4i}(X)$ be the element characterized by the requirement that
\begin{equation}
\int_X {\mathsf q}_{2i}\cdot \xi^{2n-2i} =c_X (2n-2i-1)!! q_X(\xi)^{n-i}
\end{equation}
for all $\xi\in H^2(X)$. In~\cite{bottini-og10}  classes ${\mathfrak g}_{2i}$ are defined: one has the relation
  ${\mathsf q}_{2i}=c_X {\mathfrak g}_{2i}$.
We let $C(c_2(X))$ be the rational number such that
\begin{equation}
\int_X c_2(X)\cdot \xi^{2n-2} =C(c_2(X)) q_X(\xi)^{n-1}
\end{equation}
for all $\xi\in H^2(X)$. A simple argument shows that
\begin{equation}\label{barracidue}
\ov{c_2(X)}=\frac{C(c_2(X)) }{ c_X (2n-3)!!}\qbeck_2.
\end{equation}
\begin{prp}[Beckmann, Markman, Bottini $+\ \epsilon$]\label{prp:eccoesse}
Suppose that $\cF$ is a sheaf with positive rank $r$ which has an  extended Mukai vector. Then $\cF$ is modular,  
i.e.~$\ov{\Delta(\cF)}$ is a multiple of $\qbeck_2$,
 and we may assume that $\wt{v}(\cF)=r\alpha+c_1(\cF)+s\beta$ where $s$ is determined by the equality
\begin{equation}\label{equaesse}
\ov{\Delta(\cF)}=\left(\frac{C(c_2(X)) r^2}{12 c_X (2n-3)!!}-2rs+q_X(c_1(\cF))\right)\qbeck_2.
\end{equation}
\end{prp}
\begin{proof}
By hypothesis there exist $x,s,\rho\in\QQ$ and $\lambda\in H^2(X)$ such that
\begin{equation}\label{chiave}
\ov{v(\cF)}=\rho T((x\alpha+\lambda+s\beta)^n).
\end{equation}
We  define a grading 
\begin{equation}
\Sym^n\wt{H}(X)=\bigoplus_{d} [\Sym^n\wt{H}(X)]_{2d}
\end{equation}
by letting $\alpha^i\omega_1\ldots\omega_k\beta^j\in  [\Sym^n\wt{H}(X)]_{2k+4j}$. The inclusion $\Psi$ maps $\SH^{2d}(X)$ 
into $[\Sym^n\wt{H}(X)]_{2d}$ and the orthogonal projection maps $[\Sym^n\wt{H}(X)]_{2d}$  
onto  $\SH^{2d}(X)$. Hence the equality matches homogeneous elements of the same degrees. Developing up to degree $4$ we get that
\begin{multline*}
r+c_1(\cF)+\ov{\ch_2(\cF)}+\frac{r C(c_2(X))}{24 c_X(2n-3)!!}\qbeck_2=\\
=\rho\cdot\left(x^n T(\alpha^n)+nx^{n-1}T(\alpha^{n-1}\lambda)+nx^{n-1}sT(\alpha^{n-1}\beta)+{n\choose 2}x^{n-2}T(\alpha^{n-2}\lambda^2)\right).
\end{multline*}
(recall the equality in~\eqref{barracidue}.) We have
\begin{equation}
T(\alpha^n)=n!,\quad T(\alpha^{n-1}\lambda)=(n-1)!\lambda
\end{equation}
because $\Psi(n!)=\alpha^n$ and $\Psi((n-1)!\lambda)=\alpha^{n-1}\lambda$. 
On the other hand by~\cite[Lemma~3.5]{beckmann-ext-muk} and~\cite[Lemma~3.8]{bottini-og10} (note that $c_X{\mathfrak g}_2=\qbeck_2$) we have
\begin{equation}
 T(\alpha^{n-1}\beta)=(n-1)!\qbeck_2,\quad 
T(\alpha^{n-2}\lambda^2)=(n-2)!(\lambda^2-q_X(\lambda)\qbeck_2).
\end{equation}
Hence we get the equality 
\begin{multline*}
r+c_1(\cF)+\ov{\ch_2(\cF)}+\frac{r C(c_2(X))}{24 c_X(2n-3)!!}\qbeck_2=\\
=\rho\cdot\left(n! x^n +n! x^{n-1}\lambda+n! x^{n-1}s\qbeck_2+\frac{n!}{2}x^{n-2}(\lambda^2-q_X(\lambda)\qbeck_2)\right).
\end{multline*}
Since $r>0$ we may choose $x=r$ and 
\begin{equation}\label{fattoreprop}
\rho=\frac{1}{n! r^{n-1}}.
\end{equation}
The proposition follows.
\end{proof}
\begin{expl}\label{expl:livorno}
Let $X$ be of type $K3^{[2]}$, and hence $\SH(X)=H(X)$. Then $c_X=1$ and $C(c_2(X))=30$, i.e.~$c_2(X)=30\qbeck_2$. Hence $\wt{v}(\cF)=r\alpha+c_1(\cF)+s\beta$ where $s$ is the solution of the equation
\begin{equation}\label{essek3[2]}
\Delta(\cF)=\left(\frac{r^2}{12}-\frac{rs}{15}+\frac{q_X(c_1(\cF))}{30}\right)c_2(X).
\end{equation}
\end{expl}
\subsection{The  four dimensional case} 
\setcounter{equation}{0}
In the present subsection we assume that  $\dim X=4$. If $\lambda\in H^2(X)$ we let $\lambda^{\vee}\in H^6(X)\cong H^2(X)^{\vee}$ be the linear form associated to $\lambda$ by $q_X$, i.e.~such that for all $\mu\in H^2(X)$ we have
\begin{equation}
\int_X\lambda^{\vee}\cdot\mu=q_X(\lambda,\mu).
\end{equation}
Let $\cF$ be a sheaf of positive rank $r$ with an extended Mukai vector $\wt{v}(\cF)$. Thus by Proposition~\ref{prp:eccoesse} we may assume that 
$\wt{v}(\cF)=r\alpha+c_1(\cF)+s\beta$ where $s$ is the solution of the linear equation in~\eqref{equaesse}. The proposition below is essentially contained in~\cite[Corollary 3.10]{bottini-og10}. More precisely the addend  $s\lambda^{\vee}/r$ in the right hand side of the 
equation in the statement of Bottini's Corollary should be multiplied by $c_X$.
\begin{prp}\label{prp:chichichi}
Keep assumptions and notation as above. Then
\begin{equation}
\ov{\ch_3(\cF)}=\left(  \frac{s c_X}{r}-\frac{C(c_2(X))}{24}\right)c_1(\cF)^{\vee}.
\end{equation}
\end{prp}
\begin{proof}
Let $\lambda\in H^2(X)$. A straightforward computation gives that
\begin{equation}
T(\lambda\beta)=c_X\lambda^{\vee},\qquad \ov{c_2(X)\lambda}=C(c_2(X))\lambda^{\vee}.
\end{equation}
Since $r\alpha+c_1(\cF)+s\beta$ is an extended Mukai vector of $\cF$ with proportionality factor $1/2r$ (see~\eqref{fattoreprop}) we have 
\begin{multline}
\ov{\ch_3(\cF)}+\frac{C(c_2(X))}{24}c_1(\cF)^{\vee}=\ov{\ch_3(\cF)}+\frac{\ov{c_2(X) c_1(\cF)}}{24}c_1(\cF)^{\vee}= \\
\ov{v_3(\cF)}=\frac{1}{2r}T\left([(r\alpha+c_1(\cF)+s\beta)^2]_6\right)=\frac{c_X s}{r} c_1(\cF)^{\vee}.
\end{multline}
The proposition follows.

\end{proof}
\begin{expl}\label{expl:haifa}
Let $X$ be of type $K3^{[2]}$, in particular $\SH(X)=H(X)$. Then the equation in Proposition~\ref{prp:chichichi} reads 
\begin{equation}\label{blistering}
\ch_3(\cF)=\left( \frac{s}{r}-\frac{5}{4}\right)c_1(\cF)^{\vee}.
\end{equation}
\end{expl}
\subsection{Computation of $\ch_3(\cG)$}\label{subsec:benedetta}
\setcounter{equation}{0}
Let $\cE_1,\cE_2$ be sheaves on a $K3$ surface $S$ satisfying the hypotheses of Proposition~\ref{prp:gimodulare}. Let 
we let $\lambda\in H^2(S;\QQ)$ be given by
\begin{equation}
\lambda\coloneq\frac{c_1(\cE_1)}{r_1}=\frac{c_1(\cE_2)}{r_2}\in  H^2(S;\QQ),\qquad d\coloneq\int_S\lambda^2.
\end{equation}
\begin{prp}\label{prp:chitregi}
Let $\cG=\cG(\cE_1,\cE_2)$. Then
\begin{equation}
\ch_3(\cG)=\frac{d-3}{2}c_1(\cG)^{\vee}.
\end{equation}
\end{prp}
\begin{proof}
We adopt the notation introduced in Section~\ref{sec:allagrande}. In particular  $r_i$ is the rank of $\cE_i$, $\eta\in H^4(S)$ is the fundamental class of $S$, and $e\in H^2(X_2(S);\ZZ)$ is the Poincar\`e dual of  the  exceptional 
divisor of the blow up map $\tau\colon X_2(S)\to S^2$.  We let $\ov{r}=r_1 r_2$ (as in Remark~\ref{rmk:cosgen}).

The exact sequence in~\eqref{gimodel} and  the GRR Theorem applied to $\iota_{1,2,*}(\cR_{1,2})$ give that
\begin{multline}\label{indianajones}
\rho^{*}\ch_3(\cG)=\ov{r}\biggl[\tau_1^{*}\lambda\cdot\tau_2^{*}(\lambda^2-2\eta)+\tau_1^{*}(\lambda^2-2\eta)\cdot\tau_2^{*}\lambda +  \\
 +e\cdot(\tau_1^{*}\eta+\tau_2^{*}\eta)-\frac{e}{2}\cdot(\tau_1^{*}\lambda+\tau_2^{*}\lambda)^2+\frac{e^2}{2}\cdot(\tau_1^{*}\lambda+\tau_2^{*}\lambda)
-\frac{e^3}{6}\biggr].
\end{multline}
(Note: one must use the equality in~\eqref{utileq}.) Next we note that we have the following relations in the cohomology of $X_2(S)$:
\begin{equation}\label{ecubo}
\frac{e^2}{2}\cdot(\tau_1^{*}\lambda+\tau_2^{*}\lambda)=-\tau_1^{*}\eta\cdot \tau_2^{*}\lambda-\tau_1^{*}\lambda\cdot \tau_2^{*}\eta,\qquad e^3=-12 e(\tau_1^{*}\eta+\tau_2^{*}\eta).
\end{equation}
(To prove them intersect both sides of the equalities with generators of $H^2(X_2(S))$.) Feeding the equalities in~\eqref{ecubo} into the equality
 in~\eqref{indianajones} one gets that
\begin{equation}
\rho^{*}\ch_3(\cG)=(d-3)\ov{r}\left[\tau_1^{*}\lambda\cdot\tau_2^{*}\eta+\tau_1^{*}\eta\cdot\tau_2^{*}\lambda
-e\cdot(\tau_1^{*}\eta+\tau_2^{*}\eta)\right].
\end{equation}
On the other hand, using the equalities in~\eqref{ecubo} one gets that
\begin{multline}
2\int_{S^{[2]}}\ch_3(\cG)\cdot(\mu(\alpha)+t\delta)=\int_{X_2(S)}\rho^{*}\ch_3(\cG)\cdot \rho^{*}(\mu(\alpha)+t\delta)= \\
=(d-3)q_{S^{[2]}}(c_1(\cG),\mu(\alpha)+t\delta).
\end{multline}
Proposition~\ref{prp:chitregi} follows.
\end{proof}
\begin{rmk}\label{rmk:deg3ok}
Let us assume  that the sheaf $\cG=\cG(\cE_1,\cE_2)$ has an extended Mukai vector $\wt{v}(\cG)$. By Example~\ref{expl:livorno} we may set
$\wt{v}(\cG):=(2\ov{r}\alpha+\ov{r}(2\mu(\lambda)-\delta)+s_{\cG}\beta)$, 
where  $s=s_{\cG}$ is the solution of the equation in~\eqref{essek3[2]} with $\cF=\cG$.   
We claim that the equality in~\eqref{blistering} holds for $\cF=\cG$.
In fact by Proposition~\ref{prp:gimodulare} we get that
\begin{equation}
s_{\cG}=\frac{\ov{r}(2d-1)}{2},
\end{equation}
and hence
\begin{equation}\label{extmukvect}
\wt{v}(\cG)=2\ov{r}\alpha+2\ov{r}\mu(\lambda)-\ov{r}\delta+\frac{\ov{r}(2d-1)}{2}\beta.
\end{equation}
A straightforward computation shows that the equality in~\eqref{blistering} holds. 
\end{rmk}
\subsection{Proof of Proposition~\ref{prp:atomicose}} 
\setcounter{equation}{0}
In order to prove Proposition~\ref{prp:atomicose} we compute $\ch_4(\cG)$. 
We adopt the notation of Subsection~\ref{subsec:benedetta}, and for $i\in\{1,2\}$ we let $v(\cE_i)=(r_i,l_i,s_i)$ be the Mukai vector of $\cE_i$.
\begin{prp}\label{prp:chiquattrogi}
Let $\cG=\cG(\cE_1,\cE_2)$. Then
\begin{equation}
\int_{S^{[2]}}\ch_4(\cG)=s_1 s_2- \frac{\ov{r}(3d-4)}{2}.
\end{equation}
\end{prp}
\begin{proof}
 The exact sequence in~\eqref{gimodel} gives that $\rho^{*}\ch_4(\cG)=\ch_4(\cF)-\ch_4(\iota_{1,2,*}\cR_{1,2})$ A straightforward computation gives that
\begin{equation}
\int_{X_2(S)}\ch_4(\cF)=2(s_1-r_1)(s_2-r_2)=2s_1 s_2-2\ov{r} d +2\ov{r}.
\end{equation}
The last equality holds because~\eqref{sommaquad} reads
\begin{equation}\label{vincolo}
\frac{s_1}{r_1}+\frac{s_2}{r_2}=d.
\end{equation}
The GRR Theorem applied to $\iota_{1,2,*}(\cR_{1,2})$ gives that
\begin{equation*}
\int_{X_2(S)}\ch_4(\iota_{*}\cR)=\int_{X_2(S)}\left(-\frac{\ov{r}e^4}{24}+\frac{r_2(s_1-r_1)}{2}+\frac{r_1(s_2-r_2)}{2}+\frac{\ov{r}d}{2}\right)=\ov{r}(d-2).
\end{equation*}
It follows that
\begin{equation}
\int_{X_2(S)}\rho^{*}\ch_4(\cG)=2s_1 s_2- \ov{r}(3d-4).
\end{equation}
This proves the proposition because $\rho$ has degree $2$.
\end{proof}
\begin{proof}[Proof of Proposition~\ref{prp:atomicose}]
Since the Verbitsky subalgebra of $S^{[2]}$ is equal to the whole cohomology algebra, 
 $\cG$ is atomic if and only if $v(\cG)=(4\ov{r})^{-1}T(\wt{v}^2)$, where $\wt{v}$ is given by~\eqref{extmukvect} (see~\eqref{fattoreprop} for the factor $(4\ov{r})^{-1}$).  By Remark~\ref{rmk:deg3ok}  the equality holds except possibly for the degree $8$ components. Hence  $\cG$ is atomic if and only if  
\begin{equation}\label{vu4previsto}
\int_{S^{[2]}}v_4(\cG)=\frac{\ov{r}(2d-1)^2}{16}\int_{S^{[2]}} T(\beta^2)=\frac{\ov{r}(2d-1)^2}{16},
\end{equation}
where $v_4(\cG)$ is the component of degree $8$ of 
the Mukai vector of $\cG$. 
We compute $v_4(\cG)$, the component of degree $8$ of 
the Mukai vector of $\cG$. We have
\begin{equation}
\sqrt{\Td(S^{[2]})}=1+\frac{5}{4}{\mathsf q}_{2}+\frac{25}{32}{\mathsf q}_{4}.
\end{equation}
The equalities in~\eqref{rangoeciuno} and~\eqref{discform} give that
\begin{equation}
\ch_2(\cG)=\frac{\ov{r}(2\mu(\lambda)-\delta)^2}{4}-\frac{\ov{r} c_2(S^{[2]})}{12}.
\end{equation}
Thus we have computed $\ch(\cG)$, and we get that 
\begin{equation*}
\int_{S^{[2]}}v_4(\cG)=s_1 s_2- \frac{\ov{r}(3d-4)}{2}+\frac{25\ov{r}}{16}+
\frac{5\ov{r}}{48}\int_{S^{[2]}}{\mathsf q}_2\cdot (3(2\mu(\lambda)-\delta)^2-c_2(S^{[2]})).
\end{equation*}
The last integral is computed by invoking the defining property of ${\mathsf q}_2$, the equalities $c_2(X)=30{\mathsf q}_2$ (see 
Example~\ref{expl:livorno}) and $\int_{S^{[2]}}c_2(X)^2=828$. Summing up, one gets that 
\begin{equation}\label{vu4vero}
\int_{S^{[2]}}v_4(\cG)=s_1 s_2-\frac{\ov{r}(4d-1)}{16}.
\end{equation}
Thus $\cG$ is atomic if and only if both the equalities in~\eqref{vu4previsto} and~\eqref{vu4vero} hold, i.e.~if and only if
\begin{equation}
s_1 s_2=\frac{\ov{r}d^2}{4}.
\end{equation}
Since the equality in~\eqref{vincolo} holds, it follows that $\cG$ is atomic if and only if $s_1/r_1=s_2/r_2$. The latter equation holds if and only if $v(\cE_1)^2=v(\cE_2)^2=0$. 
\end{proof}
\section{Modular sheaves and stability}\label{sec:ancorastab}
\subsection{Main results}
\setcounter{equation}{0}
In the present section we note that the results in~\cite{ogfascimod} on variation of slope semistability of modular sheaves  with respect to polarizations hold also when considering slope semistability with respect to K\"ahler classes. We also extend  the results in~\cite{ogfascimod} on suitable polarizations of 
 Lagrangian fibrations in order to deal with sheaves whose restriction to a general Lagrangian fiber is slope semistable but not stable.

\subsection{Variation of stability  with respect to K\"ahler classes}\label{subsec:varkal}
\setcounter{equation}{0}
Let $X$ be a compact K\"ahler manifold of dimension $m$. Let $\cK(X)\subset H^{1,1}_{\RR}(X)$ be the K\"ahler cone 
(whose elements are the cohomology classes of K\"ahler metrics).  Let  $\omega\in\cK(X)$. If $\cA$ is a (non zero)  torsion free sheaf on $X$ the \emph{$\omega$ slope of $\cA$} is given by
\begin{equation}
\mu_{\omega}(\cA)\coloneq\frac{c_1(\cA)\cdot\omega^{m-1}}{\rk(\cA)}.
\end{equation}
A torsion free sheaf $\cF$ on $X$ is \emph{$\omega$ slope semistable} if  for all non zero subsheaves $\cH\subset\cF$ with $0<\rk(\cH)<\rk(\cF)$ we have
\begin{equation}\label{disegpen}
\mu_{\omega}(\cH)\le \mu_{\omega}(\cF),
\end{equation}
and it is  \emph{$\omega$ slope stable} if strict inequality holds for all such $\cH$.
If
$\cH,\cF$  are sheaves on an irreducible smooth  variety $X$ we let
\begin{equation}\label{pecora}
\lambda_{\cH,\cF}:=r(\cF) c_1(\cH)-r(\cH) c_1(\cF).
\end{equation}
The proof of Lemma~3.7 in~\cite{ogfascimod} extends with no changes in the more general framework.
\begin{lmm}\label{lmm:comesup}
Let $X$ be a HK manifold, and let  $\omega$ be a K\"ahler class on $X$.
Let $\cH,\cF$  be non zero torsion free sheaves on  $X$.
Then
$\mu_{\omega}(\cH)\ge\mu_{\omega}(\cF)$ if and only if 
\begin{equation}\label{disepend}
q_X\left(\lambda_{\cH,\cF},{\omega}\right)\ge 0.
\end{equation}
Moreover equality in~\eqref{disepend} holds if and only if $\mu_{\omega}(\cH)=\mu_{\omega}(\cF)$.
\end{lmm}
Let $\cF$ be a torsion-free modular sheaf on   a HK manifold $X$. 
We define a
 decomposition of  $\cK(X)$ 
 into walls and chambers  related to slope stability of $\cF$. 
\begin{dfn}
Let ${\mathsf a}$ be a positive real number. An \emph{${\mathsf a}$-wall} of $\cK(X)$ is the intersection 
$\lambda^{\bot}\cap \cK(X)$, where 
  $\lambda\in H^{1,1}_{\ZZ}(X)$ is a class such that  
$ -{\mathsf a} \le q_X(\lambda)< 0$
(orthogonality is with respect to the BBF quadratic form $q_X$).
\end{dfn}
As is well-known, the set of ${\mathsf a}$-walls  is  locally finite, in particular the union of all the ${\mathsf a}$-walls  is closed in $\cK(X)$. 
\begin{dfn}
An \emph{open ${\mathsf a}$-chamber} of $\cK(X)$  is a connected component of the complement of   the union of all the 
${\mathsf a}$-walls of $\cK(X)$. A K\"ahler class is ${\mathsf a}$-generic if it belongs to an open ${\mathsf a}$-chamber.
\end{dfn}
\begin{prp}\label{prp:campol}
Let $X$ be a  HK manifold, and let $\omega_0,\omega_1\in\cK(X)$.
Suppose that $\cF$ is a   torsion free modular sheaf  on $X$ which is $\omega_0$ slope stable and not $\omega_1$ slope stable. 
Then there exists a real $t$ with $0<t\le 1$ such that $t\omega_0 +(1-t)\omega_1$ belongs to an 
 ${\mathsf a}(\cF)$-wall. 
\end{prp}
\begin{proof}
One needs to prove the versions of Propositions~3.8 and~3.10 in~\cite{ogfascimod}  that one gets upon replacing  the ample cone by the K\"ahler cone. We show how to adapt the proofs   in the present context. 
By  Lemma~6.2 in~\cite{greb-toma-cmpct-mod-sheaves} (the proof
 is in~\cite{greb-toma-cmpct-mod-sheaves-arxiv})  there exists  a real $t$ with $0<t\le 1$ such that, letting 
 $\omega_t\coloneq t\omega_0+(1-t)\omega_1$, the sheaf 
 $\cF$ is strictly $\omega_t$ slope semistable, i.e.~$\omega_t$  slope semistable but not $\omega_t$  slope stable. Hence there exists an exact sequence
of torsion free non zero sheaves
\begin{equation}\label{aeffebi}
0\lra \cA\lra \cF\lra \cB\lra 0
\end{equation}
with $0<\rk(\cE)<\rk(\cF)$ and $\mu_{\omega_t}(\cA)=\mu_{\omega_t}(\cF)$, i.e.~(by Lemma~\ref{lmm:comesup})
\begin{equation}\label{vittoria}
q_X\left(\lambda_{\cA,\cF},{\omega_t}\right)= 0
\end{equation}
We may assume that $\cB$ is torsion free (note that $\cA$   is torsion free because $\cF$ is   torsion free). Moreover both $\cA$ and $\cB$ are $\omega_t$  slope semistable.
By Theorem~1.1 in~\cite{lizhangzhang} it follows that Bogomolov's inequality holds for $\cA$ and $\cB$, i.e.~that
  $\Delta(\cA)\cdot\omega^{2n-2}\ge 0$ and 
$\Delta(\cB)\cdot\omega^{2n-2}\ge 0$, where $2n$ is the  dimension of $X$.
To be precise  Theorem~1.1 in loc.cit.~states that Bogomolov's inequality  holds for slope semistable reflexive sheaves. 
From this one gets Bogomolov's inequality for a torsion free slope semistable sheaf $\cH$ arguing as follows. 
 We have a canonical exact sequence
\begin{equation*}
0\lra \cH\lra\cH^{\vee\vee}\lra\cQ\lra 0,
\end{equation*}
where $\cQ$ is supported on an analytic subspace of codimension at least $2$. Since $\cH$ is slope semistable, so is 
the double dual $\cH^{\vee\vee}$. Since the double dual is reflexive we have $\Delta(\cH^{\vee\vee})\cdot\omega^{2n-2}\ge 0$ by Theorem~1.1 in~\cite{lizhangzhang}. 
Since 
$c_1(\cH)=c_1(\cH^{\vee\vee})$ and $c_2(\cH)=c_2(\cH^{\vee\vee})+Z$, where $Z$ is an effective codimension-$2$ cycle supported on the codimension-$2$ components of $\supp(\cQ)$ we get that 
$ \Delta(\cH)\cdot\omega^{2n-2}\ge 0$. Since  Bogomolov's inequality holds for $\cA$ and $\cB$, the proof 
of Proposition~3.10 in~\cite{ogfascimod} extends to our case and hence we get that 
\begin{equation}\label{scarlatti}
-{\mathsf a}(\cF)\le q_X\left(\lambda_{\cA,\cF}\right)\le 0.
\end{equation}
Suppose that $q_X\left(\lambda_{\cA,\cF}\right)= 0$. Then $\lambda_{\cA,\cF}=0$ and it follows that 
$\mu_{\omega}(\cA)=\mu_{\omega}(\cF)$ for all K\"ahler classes $\omega$. By the exact sequence in~\eqref{aeffebi} this contradicts the assumption that 
$\cF$ is  $\omega_0$ slope stable. Thus $q_X\left(\lambda_{\cA,\cF}\right)< 0$, and hence $\lambda_{\cA,\cF}^{\bot}\cap\cK(X)$ is an ${\mathsf a}(\cF)$-wall.  We are done by~\eqref{vittoria}.
\end{proof}
\begin{crl}\label{crl:campol}
Let $X$ be a  HK manifold, and let  $\cF$ be a   torsion free modular sheaf  on $X$. Then the following hold:
\begin{enumerate}
\item
Suppose that $\omega$ is a K\"ahler class on $X$ which is ${\mathsf a}(\cF)$-generic. If  $\cF$ is  strictly 
$\omega$ slope semistable  there exists an exact sequence
of torsion free non zero sheaves
\begin{equation}\label{cetraro}
0\lra \cA\lra \cF\lra \cB\lra 0
\end{equation}
such that $r(\cF) c_1(\cA)-r(\cA) c_1(\cF)=0$.
\item
Suppose that $\omega_0,\omega_1$ are K\"ahler classes on $X$  belonging to the same open ${\mathsf a}(\cF)$-chamber. Then $\cF$ is 
 $\omega_0$ slope-stable if and only if it is $\omega_1$ slope-stable. 
\end{enumerate}
\end{crl}
\begin{proof}
 Item~(1) follows from the proof of Proposition~\ref{prp:campol}. In fact let~\eqref{cetraro} be an exact sequence with $0<\rk(\cA)<\rk(\cF)$ and 
$\mu_{\omega}(\cA)=\mu_{\omega}(\cF)$, i.e.~$q_X\left(\lambda_{\cA,\cF},\omega\right)= 0$. Then the inequalities in~\eqref{scarlatti} hold, and hence 
$\lambda_{\cA,\cF}=0$ because  
$\omega_0$, $\omega_1$   belong to the same open $a(\cF)$-chamber.  
 Item~(2) follows from the statement of Proposition~\ref{prp:campol} because $\omega_0$, $\omega_1$   belong to the same open ${\mathsf a}(\cF)$-chamber.
\end{proof}
\subsection{Modular sheaves on  Lagrangian fibrations}\label{subsec:lagfib}
\setcounter{equation}{0}
Let $\pi\colon X\to \PP^n$ be a Lagrangian fibration of a HK  manifold of dimension $2n$. We let
\begin{equation}\label{classeffe}
f:=\pi^{*}c_1(\cO_{\PP^n}(1)).
\end{equation}
Let $\cF$ be a  sheaf on $X$. If $t\in\PP^n$ we let $X_t=\pi^{-1}(t)$ and  $\cF_t:=\cF_{|X_t}$. Whenever we consider a \lq\lq general $t\in\PP^n$\rq\rq\ we may (and will) assume  that $X_t$ is smooth. 
\begin{rmk}\label{rmk:slopeonlagr}
Suppose that $X_t$ is smooth. Then the image of the restriction map 
$H^2(X;\RR)\to H^2(X_t;\RR)$ is of dimension $1$ and is generated by the class of an ample  class $\theta_t\in H^{1,1}_{\ZZ}(X_t)$, see~\cite{wieneck1}; by slope  (semi)stability of a sheaf on $X_t$ we mean stability with respect to $\theta_t$. 
\end{rmk}
\begin{dfn}\label{dfn:poladatta}
Let $\pi\colon X\to \PP^n$ be a Lagrangian fibration of a HK  manifold of dimension $2n$, and let
  ${\mathsf a}> 0$. A polarization $h$ of $X$ is \emph{${\mathsf a}$-suitable with respect to $\pi$} (or simply ${\mathsf a}$-suitable whenever there is no ambiguity regarding the fibration) if the following holds.  If 
 $\lambda\in H^{1,1}_{\ZZ}(X)$  is such that $ -{\mathsf a}\le q_X(\lambda)<0$, 
then
\begin{enumerate}
\item
$q_X(\lambda,h)>0$ implies that $q_X(\lambda,f)\ge 0$,
\item
 $q_X(\lambda,h)=0$  implies that $q_X(\lambda,f)= 0$, and
\item
 $q_X(\lambda,h)<0$  implies that $q_X(\lambda,f)\le 0$.
\end{enumerate}
\end{dfn}
\begin{rmk}
In~\cite[Def.~3.5]{ogfascimod}  a polarization $h$ is ${\mathsf a}$-suitable if it is ${\mathsf a}$-suitable according to Definition~\ref{dfn:poladatta} and in addition $q_X(\lambda,f)=0$ implies that $q_X(\lambda,h)=0$ (where $\lambda$ is as in Definition~\ref{dfn:poladatta}).
 To avoid confusion let us say that $h$ is \emph{strongly ${\mathsf a}$-suitable} if it is  ${\mathsf a}$-suitable according to~\cite[Def.~3.5]{ogfascimod}. This definition is useful  only if the Picard rank $\rho(X)$ is $2$. In fact if $\rho(X)>2$ and 
 ${\mathsf a}\gg 0$ then there is no strongly ${\mathsf a}$-suitable polarization because the 
quadratic form $q_X$ defines a negative definite quadratic form on the non zero quotient $f^{\bot}/\ZZ f$ and hence a general
 $\lambda\in f^{\bot}$ has negative square and non zero intersection with $h$.  
\end{rmk}
\begin{prp}\label{prp:lagstab}
Let $\pi\colon X\to\PP^n$ be a  Lagrangian fibration of a HK variety. Let ${\mathsf a}>0$, and let $h\in\Amp(X)$ 
 be an ${\mathsf a}$-suitable polarization. Let  $\cF$ be  a torsion free  modular sheaf  on $X$ such that ${\mathsf a}(\cF)\le {\mathsf a}$.  
\begin{enumerate}
\item[(a)]
 Suppose that $\cF$ is not $h$-slope stable, and that   $\cF_t$ is slope semistable for general $t\in\PP^n$ 
 (see Remark~\ref{rmk:slopeonlagr}). Then there exists a subsheaf  $\cH\subset \cF$ with $0<r(\cH)<r(\cF)$ such that $\mu_h(\cH)\ge\mu_h(\cF)$ and $\mu(\cH_t)=\mu(\cF_t)$ for general $t\in\PP^n$.
\item[(b)]
If  $\cF$ is $h$-slope stable, then $\cF_t$ is slope semistable 
 for general $t\in\PP^n$ (see Remark~\ref{rmk:slopeonlagr}). 
\end{enumerate}
\end{prp}
\begin{crl}\label{crl:lagstab}
Let $\pi\colon X\to\PP^n$ be a  Lagrangian fibration of a HK variety. Let ${\mathsf a}>0$, and let $h\in\Amp(X)$ 
 be an ${\mathsf a}$-suitable polarization. Let  $\cF$ be  a torsion free  modular sheaf  on $X$ such that 
 ${\mathsf a}(\cF)\le {\mathsf a}$.  
If $\cF_t$ is slope stable for general $t$
then  $\cF$ is $h$ slope stable. 
\end{crl}
\begin{rmk}
Corollary~\ref{crl:lagstab} is the same as Item~(i) of~\cite[Prop.~3.6]{ogfascimod}, but with the \lq\lq correct\rq\rq\ definition of $a$-suitability. 
\end{rmk}
 Before proving  Proposition~\ref{prp:lagstab} we go through a few  preliminaries.
The result below is somewhat technical.
\begin{lmm}[K.~Yoshioka]\label{lmm:stimakota}
Let $X$ be a hyperk\"ahler manifold. Let $h,f\in H^{1,1}_{\ZZ}(X)$ be such that $q_X(h)>0$ and  $q_X(f)=0$. Suppose that 
$\lambda\in H^{1,1}_{\ZZ}(X)$ and that
\begin{equation}
q_X(\lambda,h)=0,\qquad q_X(\lambda,f)\not=0.
\end{equation}
Then
\begin{equation}\label{terrekobe}
q_X(\lambda)\le-\frac{q_X(h)}{q_X(h,f)^2}.
\end{equation}
\end{lmm}
\begin{proof}
The proof is analogous to the proof of~\cite[Lemma~1.1]{yoshi-ellsurf}. Write $\lambda=ah+bf+\xi$ where 
$\xi\in\{h,f\}^{\bot}\cap H^{1,1}_{\ZZ}(X)$. Since $q_X$ is negative definite on $h^{\bot}\cap H^{1,1}_{\ZZ}(X)$ 
\begin{equation}\label{bellamora}
q_X(\lambda)\le q_X(ah+bf)=a^2 q_X(h)+2ab q_X(h,f)=-a^2 q_X(h),
\end{equation}
where the last equality follows from $0=q_X(\lambda,h)=a q_X(h)+bq_X(h,f)$. By hypothesis $q_X(\lambda,f)=aq_X(h,f)$ is non zero, and since it is an integer we get that $a^2q_X(h,f)^2\ge 1$, i.e.~$a^2\ge q_X(h,f)^{-2}$. Plugging this inequality in~\eqref{bellamora} we get 
the inequality in~\eqref{terrekobe}.
\end{proof}
The next results  guarantee the existence of $a$-suitable polarizations.
\begin{prp}\label{prp:suitif}
Let ${\mathsf a}>0$. Let $\pi\colon X\to \PP^n$ be a Lagrangian fibration of a HK manifold of dimension $2n$. Let $h$ be a polarization of $X$. If 
  $ q_X(h)> {\mathsf a}\cdot q_X(h,f)^2$ 
 then $h$ is  ${\mathsf a}$-suitable.
\end{prp}
\begin{proof}
Let  $\lambda\in H^{1,1}_{\ZZ}(X)$ be as in Definition~\ref{dfn:poladatta}. Suppose that $q_X(\lambda,h)=0$.  Then $q_X(\lambda,f)=0$ by Lemma~\ref{lmm:stimakota}. To finish the proof we assume that $q_X(\lambda,h)$, $q_X(\lambda,f)$ are non zero of opposite signs and we get a contradiction.   
Let  
$m_0\coloneq q_X(\lambda,h)$ and $n_0\coloneq -q_X(\lambda,f)$. Let $h_0\coloneq m_0f+n_0h$. Then $q_X(\lambda,h_0)=0$. 
Since $m_0,n_0$ are both non zero of the same sign we have $q_X(h_0)>0$ and $q_X(h_0,f)\not=0$. By Lemma~\ref{lmm:stimakota} (with $h=h_0$) we get that
\begin{equation}\label{terrekobe}
q_X(\lambda)\le-\frac{q_X(h_0)}{q_X(h_0,f)^2}=-\frac{q_X(h)}{q_X(h,f)^2}-\frac{2m_0}{n_0 q_X(h,f)}<-{\mathsf a}.
\end{equation}
This is a contradiction.
\end{proof}
\begin{crl}[R.~Friedman]\label{crl:suitif}
Let $\pi\colon X\to \PP^n$ be a Lagrangian fibration of a HK manifold of dimension $2n$. Let $h_0$ be a polarization of $X$ and let ${\mathsf a}> 0$.
Let $N\in\NN$ be such that  $2N> {\mathsf a}\cdot q_X(h_0,f)$.
 Then $h_0+Nf$ is an ${\mathsf a}$-suitable polarization.
\end{crl}
\begin{proof}
Clearly $h_0+Nf$ is a polarization, and since
\begin{equation*}
q_X(h_0+Nf)=q_X(h_0)+2N\cdot q_X(h_0,f)>{\mathsf a}\cdot q_X(h_0,f)^2={\mathsf a}\cdot q_X(h_0+Nf,f)^2,
\end{equation*}
it is ${\mathsf a}$-suitable by Proposition~\ref{prp:suitif}.
\end{proof}
\begin{proof}[Proof of Proposition~\ref{prp:lagstab}]
(a): Let $\cH\subset\cF$ be a subsheaf. We recall  that
\begin{equation}\label{segnoacca}
\text{$\mu_h(\cH)-\mu_h(\cF)$ and $q_X(\lambda_{\cH,\cF},h)$ have the same sign or are both $0$.} 
\end{equation}
 (See~\cite[Lemma~3.7]{ogfascimod}.) 
Moreover if $t\in\PP^n$ is general we have
\begin{equation}\label{segnoeffe}
\text{$\mu_h(\cH_t)-\mu_h(\cF_t)$ and $q_X(\lambda_{\cH,\cF},f)$ have the same sign or are both $0$.} 
\end{equation}
In fact~\eqref{segnoeffe} follows from the equalities 
\begin{multline}\label{papamorto}
r(\cH) r(\cF)(\mu(\cH_t)-\mu(\cF_t))=\int_{X_t} \lambda_{\cH_t,\cF_t}\cdot h_t^{n-1}=\int_X \lambda_{\cH,\cF}\cdot h^{n-1}\cdot f^n= \\
=n! c_X \cdot q_X(h,f)^{n-1} \cdot q_X(\lambda_{\cH,\cF},f),
\end{multline}
and  positivity of $c_X$, $q_X(h,f)$. 
By~\eqref{segnoacca}, \eqref{segnoeffe} it suffices to prove that  there exists a subsheaf $\cH\subset\cF$  such that  $0<r(\cH)<r(\cF)$ and $q_X(\lambda_{\cH,\cF},h)\ge 0$, $q_X(\lambda_{\cH,\cF},f)=0$.

First assume  that $\cF$ is $h$-slope semistable.  Hence there exists a subsheaf $\cH\subset\cF$  such that  $0<r(\cH)<r(\cF)$ and 
$q_X(\lambda_{\cH,\cF},h)= 0$. By~\cite[Prop.~3.10]{ogfascimod} we get that $-{\mathsf a}(\cF)\le q_X(\lambda_{\cH,\cF})\le 0$. 
If $\lambda_{\cH,\cF}= 0$ then $q_X(\lambda_{\cH,\cF},f)=0$ trivially, if $\lambda_{\cH,\cF}\not= 0$ then 
$q_X(\lambda_{\cH,\cF})< 0$ because the restriction of $q_X$ to $H^{1,1}_{\ZZ}(X)$ has signature $(1,\rho(X))$, and hence 
$q_X(\lambda_{\cH,\cF},f)=0$ because $h$ is ${\mathsf a}$-suitable.

Next assume  that $\cF$ is not $h$-slope semistable. Thus there exists $\cG\subset\cF$  such that  $0<r(\cG)<r(\cF)$ and 
$q_X(\lambda_{\cG,\cF},h)> 0$. If $q_X(\lambda_{\cG,\cF},f)=0$ we are done, hence we may assume that 
$q_X(\lambda_{\cG,\cF},f)\not=0$. Then $q_X(\lambda_{\cG,\cF},f)<0$ by~\eqref{segnoeffe}. 
Let $\mathsf S$ be the set  of rational numbers $s\in(0,1)$ for which there exists a subsheaf $\cH\subset\cF$,  with $0<r(\cH)<r(\cF)$,  such that 
\begin{equation}\label{comenovem}
q_X(\lambda_{\cH,\cF},(1-s)h+s f)=0.
\end{equation}
Then $\mathsf S$ is non  empty because there exists a  rational number $s\in(0,1)$ for which~\eqref{comenovem} holds with $\cH=\cG$.
 We claim that $\mathsf S$  is finite. In fact if~\eqref{comenovem} holds with $s\in(0,1)$ then $q_X(\lambda_{\cH,\cF},h)\ge 0$, because 
 $q_X(\lambda_{\cH,\cF},f)\le 0$ by~\eqref{segnoeffe}. 
  Since the set of subsheaves  $\cH\subset\cF$ such that 
 $\mu_h(\cH)\ge\mu_h(\cF)$  is bounded it follows that $\mathsf S$  is finite. Hence $\mathsf S$ has a maximum 
 $s_{*}$. Let $h_{*}\coloneq (1-s_{*})h+s_{*} f$. 
 
 Suppose that $\cF$ is not $h_{*}$ slope semistable. Then there exists a subsheaf 
  $\cH\subset\cF$    with $0<r(\cH)<r(\cF)$  such that $q_X(\lambda_{\cH,\cF},h_{*})>0$. If $q_X(\lambda_{\cH,\cF},f)<0$ then there exists $s\in (s_{*},1)$ such that~\eqref{comenovem} holds, and this is a contradiction because   $s_{*}$ is the maximum of    $\mathsf S$. Thus  $q_X(\lambda_{\cH,\cF},f)= 0$ by~\eqref{segnoeffe}.
   Since $q_X(\lambda_{\cH,\cF},h_{*})>0$ and $q_X(\lambda_{\cH,\cF},f)= 0$ we get $q_X(\lambda_{\cH,\cF},h)>0$, and we are done. 

 Suppose that $\cF$ is  $h_{*}$ slope semistable.  Then there exists a subsheaf 
  $\cH\subset\cF$    with $0<r(\cH)<r(\cF)$  such that $q_X(\lambda_{\cH,\cF},h_{*})= 0$. We claim that 
  $q_X(\lambda_{\cH,\cF},f)=0$. Granting this for the moment being, we get that $q_X(\lambda_{\cH,\cF},h)= 0$ because
  $q_X(\lambda_{\cH,\cF},h_{*})= 0$, and hence we are done. We finish by proving that  $q_X(\lambda_{\cH,\cF},f)=0$.
   By~\cite[Prop.~3.10]{ogfascimod} we get that $-{\mathsf a}(\cF)\le q_X(\lambda_{\cH,\cF})\le 0$. If $\lambda_{\cH,\cF}= 0$ then $q_X(\lambda_{\cH,\cF},f)=0$ trivially.
If  $\lambda_{\cH,\cF}\not= 0$ then 
$q_X(\lambda_{\cH,\cF})< 0$. Suppose that   $q_X(\lambda_{\cH,\cF},f)\not=0$. Then $q_X(\lambda_{\cH,\cF},f)<0$ by~\eqref{segnoeffe}, and hence $q_X(\lambda_{\cH,\cF},h)< 0$ because $h$ is ${\mathsf a}$-suitable. This contradicts the equality 
$q_X(\lambda_{\cH,\cF},h_{*})= 0$. 
 
(b):  Item~(ii) of~\cite[Prop.~3.6]{{ogfascimod}} is the same exact statement. The proof there works as well with our new definition of
 ${\mathsf a}$-suitable polarization.  

\end{proof}
If $\dim X=2$, then one can prove a stronger version of Item~(b) of Proposition~\ref{prp:lagstab}. 
\begin{prp}\label{prp:ellstab}
Let $\pi\colon X\to\PP^1$ be an elliptic   fibration of a $K3$ surface.  Let $\cF$ be  a torsion free sheaf   on $X$, and let $h\in H^{1,1}_{\ZZ}(X)$ be an ample  class  which is ${\mathsf a}(\cF)$-suitable.  If $\cF$ is $h$-slope semistable, then $\cF_t$ is semistable 
 for general $t\in\PP^1$. 
\end{prp}
\begin{proof}
If $\cF$ is $h$-slope stable,  then $\cF_t$ is slope semistable 
 for general $t\in\PP^1$ by Proposition~\ref{prp:lagstab} (every torsion free sheaf   on a $K3$ surface is modular). 
 Now suppose that 
$\cF$ is strictly $h$-slope semistable, and let
\begin{equation}
0=\cG_0\subsetneq \cG_1\subsetneq \cG_2 \subsetneq \ldots \subsetneq \cG_m=\cF
\end{equation}
be a (slope) Jordan-H\"older filtration of $\cF$ with $\cG_i/\cG_{i-1}$ torsion free for all $i$. This means that for  $i\in\{1,\ldots,m\}$
the quotient $\cG_i/\cG_{i-1}$ is $h$ slope stable and
$\mu_h(\cG_i)=\mu_h(\cF)$.
Let $i\in\{1,\ldots,m-1\}$. Then $q_X(\lambda_{\cG_i,\cF},h)=0$ and, by Proposition~3.10   in~\cite{ogfascimod}, we have 
$-{\mathsf a}(\cF)\le q_X(\lambda_{\cG_i,\cF})\le 0$ with $q_X(\lambda_{\cG_i,\cF},h)= 0$ only if $\lambda_{\cG_i,\cF}= 0$. Since $h$ is ${\mathsf a}(\cF)$-suitable, we get that $q_X(\lambda_{\cG_i,\cF},f)=0$. 

Let $t\in\PP^1$ be a general point. 
From Equation~(3.3.2) in~\cite{ogfascimod} (note that, in the notation of that equation, we have $\lambda_{\cE,\cF}^2\le 0$) one gets  that  
${\mathsf a}(\cG_i/\cG_{i-1})\le {\mathsf a}(\cF)$, and hence $h$ is ${\mathsf a}(\cG_i/\cG_{i-1})$-suitable. By Proposition~\ref{prp:lagstab} it follows that the restriction 
of $\cG_i/\cG_{i-1}$ to $X_t$ is  slope semistable. Moreover $\mu(\cG_{i|X_t})=\mu(\cF_{|X_t})$ because $q_X(\lambda_{\cG_i,\cF},f)=0$.  
Hence we get a filtration
\begin{equation}
0\not=\cG_{1|X_t}\subsetneq \cG_{2|X_t} \subsetneq \ldots \subsetneq \cG_{m|X_t}=\cF_t
\end{equation}
where each term has the same slope, and each successive quotient is semistable. It follows that $\cF_t$ is slope-semistable. 
\end{proof}
\section{A component of $M_{w_0}(S^{[2]},h_{S^{[2]}})$ birational to $\cM_{v_2}(S,h_S)$}\label{sec:modbir}
\subsection{Main result}\label{subsec:spalmo}
\setcounter{equation}{0}
In the present section $S$ is a $K3$ surface 
with an elliptic fibration 
$\varepsilon\colon S\to \PP^1$  as in Claim~\ref{clm:neronsevero}. Recall that this means the following. Letting $C\subset S$ 
be a  elliptic fiber we have 
\begin{equation}
\NS(S)=\ZZ[D]\oplus\ZZ[C], \quad
D\cdot D=2m_0,\quad D\cdot C=d_0,
\end{equation}
where $m_0,d_0$ are positive natural numbers. Moreover $m_0,d_0$ can be assigned arbitrarily.
The  \emph{associated Lagrangian fibration} of $S^{[2]}$ is given by
\begin{equation}\label{fiblagk3n}
\begin{matrix}
S^{[2]} & \overset{\pi}{\lra} &  (\PP^1)^{(2)}\cong \PP^2 \\
[Z] & \mapsto & \sum_{p\in S}\ell(\cO_{Z,p})\varepsilon(p)
\end{matrix}
\end{equation}
\begin{ass-dfn}\label{ass-dfn:numeretti}
Keeping assumptions as above, suppose that
$r_1,a$  are positive integers such that 
\begin{equation}\label{condizioni}
r_1\mid 2a,\qquad r_1\mid(m_0+1),\qquad \gcd(r_1,d_0)=1.
\end{equation}
 Let
\begin{equation}\label{verovudue}
v_1:=\left(r_1,D,\frac{m_0+1}{r_1}\right),\qquad v_2:=av_1-\frac{2a}{r_1}\left(0,0,1\right),
\end{equation}
and
\begin{equation}\label{doppiovuzero}
w_0:=ar_1\left(2 r_1, 2\mu(D)-r_1\delta, \frac{a r_1^3 c_2(S^{[2]})}{3} \right).
\end{equation}
\end{ass-dfn}
Let $h_S$ be a polarization of $S$ which is ${\mathsf a}(v_2)$-suitable, and let $h_{S^{[2]}}$ be a polarization of $S^{[2]}$  which is ${\mathsf a}(w_0)$-suitable.
  In the present section we show that the moduli space  $\cM_{v_2}(S,h_S)$ is birational to an  irreducible component of  $M_{w}(S^{[2]},h_{S^{[2]}})$ if $a\ge 2$. In order to formulate our result  more precisely we note the following. Since $v_1^2=-2$ there is a unique $h_S$ stable sheaf $\cE_1$ with $v(\cE_1)=v_1$, and it is locally free. 
Since $v_1^2=2a^2+2$ the moduli space $\cM_{v_2}(S,h_S)$ is  irreducible of dimension $2a^2+2$. Let $[\cE_2]\in\cM_{v_2}(S,h_S)$ be a  general point. Then $\cE_2$ is locally free and hence $\cG(\cE_1,\cE_2)$ is locally free. By 
Assumption-Definition~\ref{ass-dfn:numeretti} and Lemma~\ref{lmm:vuquad} we have
\begin{equation}\label{lugano}
w(\cG(\cE_1,\cE_2))=w_0.
\end{equation}
\begin{thm}\label{thm:dadueauno}
Suppose that Assumption-Definition~\ref{ass-dfn:numeretti} holds, and that $a\ge 2$. Let $h_S$ be a polarization of $S$ such which is ${\mathsf a}(v_2)$-suitable, and let  $h_{S^{[2]}}$ be a polarization of $S^{[2]}$ 
which is ${\mathsf a}(w_0)$-suitable. Let $[\cE_2]\in\cM_{v_2}(S,h_S)$ be a general point. Then the 
 locally free sheaf $\cG(\cE_1,\cE_2)$ is $h_{S^{[2]}}$ slope stable. The rational map
\begin{equation}\label{modulimoduli}
\begin{matrix}
\cM_{v_2}(S,h_S) & \overset{\varphi}{\dra} & M_{w}(S^{[2]},h_{S^{[2]}}) \\
[\cE_2] & \mapsto & [\cG(\cE_1,\cE_2)]
\end{matrix}
\end{equation}
is birational onto an irreducible component $M_{w_0}(S^{[2]},h_{S^{[2]}})^{\bullet}$ of $M_{w_0}(S^{[2]},h_{S^{[2]}})$. 
\end{thm}
\subsection{Stability of $\cG(\cE_1,\cE_2)$ }
\setcounter{equation}{0}
\begin{prp}\label{prp:coppiadie}
Suppose that Assumption-Definition~\ref{ass-dfn:numeretti} holds. Let $h_S$ be a polarization of $S$ which is ${\mathsf a}(v_2)$-suitable. Then the following hold:
\begin{enumerate}
\item[(a)]
The restriction of $\cE_1$ to any  fiber  of the elliptic fibration with elliptic fiber $C$ is slope stable.
\item[(b)]
Suppose that $a\ge 2$. Let  $[\cE_2]\in M_{v_2}(S,h_S)$ be a general point and let $C$ be a general elliptic fiber.  The restriction 
$\cE_{2|C}$ is semistable, with pairwise non isomorphic Jordan-H\"older (JH) addends, none of which is isomorphic to 
$\cE_{1|C}$. 
\end{enumerate}
\end{prp}
\begin{proof}
The polarization $h_S$  is ${\mathsf a}(v_1)$-suitable because ${\mathsf a}(v_1)<{\mathsf a}(v_2)$. 
 Hence    Item~(a) holds by Proposition~6.2 in~\cite{ogfascimod}. 
Since $a\ge 2$ and $[\cE_2]\in M_{v_2}(S,h_S)$ is a general point 
 the decomposition curve $\Lambda(\cE_2)$    is integral and smooth by Corollary~\ref{crl:spettroirrid}. It follows that if $x\in\PP^1$ is a general point the intersection $C_x\cap \Lambda(\cE_2)=\Lambda(\cE_2)_x$ consists of  $a$ distinct points, and hence the JH addends of $\cE_{2|C_x}$
 are pairwise non isomorphic. Moreover $C_x\cap \Lambda(\cE_2)\not=C_x\cap \Sigma$ (for all  $x\in\PP^1$ because $\Lambda(\cE_2)\cdot\Sigma=0$), and hence none of 
  the JH addends is isomorphic to $\cE_{1|C}$. 

\end{proof}
\begin{prp}\label{prp:stabonlagr}
Let hypotheses be as in Theorem~\ref{thm:dadueauno}. If $[\cE_2]\in M_{v_2}(S,h_S)$ is a general point
 then  $\cG(\cE_1,\cE_2)$ is an $h_{S^{[2]}}$ slope stable locally free sheaf.
\end{prp} 
\begin{proof}
 Let $\cG=\cG(\cE_1,\cE_2)$. 
Let $x\not= y\in\PP^1$, and let $C_x,C_y$ be the corresponding  fibers 
 of the elliptic fibration $S\to\PP^1$. Then we have an identification
\begin{equation}\label{fibralagrgen}
\pi^{-1}(x+y)=C_x\times C_y.
\end{equation}
 Letting $t\coloneq x+y$, we have
\begin{equation}\label{prodtens}
\cG_t\coloneq\cG_{|\pi^{-1}(t)}=\left(\cE_{1|C_x}\right)\boxtimes \left(\cE_{2|C_y}\right)\oplus 
\left(\cE_{2|C_x}\right)\boxtimes \left(\cE_{1|C_y}\right).
\end{equation}
By Proposition~\ref{prp:coppiadie} both $\cE_{1|C_x}$ and $\cE_{1|C_y}$ are stable. Suppose in addition that $x\not= y$ are general. By Proposition~\ref{prp:coppiadie} both
$\cE_{2|C_x}$ and $\cE_{2|C_y}$ are  semistable with pairwise non isomorphic JH addends 
$V_1(x),\ldots,V_a(x)$ and $V_1(y),\ldots,V_a(y)$. Moreover no $V_i(x)$ is isomorphic to $\cE_{1|C_x}$ and no $V_j(y)$ is isomorphic to 
$\cE_{1|C_y}$. 
The restriction of the polarization $h_{S^{[2]}}$ to $C_x\times C_y$ is of  product type. It follows that for $i,j\in\{1,\ldots,a\}$ the tensor products
\begin{equation}\label{listaprod}
\left(\cE_{1|C_x}\right)\boxtimes V_j(y),\quad    
V_i(x)\boxtimes \left(\cE_{1|C_x}\right)
\end{equation}
are slope stable, see Proposition~6.10 in~\cite{ogfascimod}, and they all have the same slope (with respect to the restriction of  $h_{S^{[2]}}$). 
The upshot is that  the left hand side of~\eqref{prodtens} is slope semistable, with pairwise non isomorphic JH addends given by the sheaves appearing in~\eqref{listaprod}.
It follows that any subsheaf $\cA\subset\cG_t$ such that $\mu(\cA)=\mu(\cG_t)$ (the slope is with respect to the restriction of  $h_{S^{[2]}}$) is a direct sum $\cA=\cA'\oplus \cA''$ where $\cA',\cA''$ are  slope semistable subsheaves of the first and second addends of the decomposition in~\eqref{prodtens}
respectively, and their JH addends are a subset of the JH addends appearing in~\eqref{listaprod}.

We are ready to show that $\cG$ is $h_{S^{[2]}}$ slope stable.  Recall that  $\cG$ is modular (see Example~\ref{expl:didigi}).     
The polarization  $h_{S^{[2]}}$ is ${\mathsf a}(\cG)$-suitable because it is ${\mathsf a}(w_0)$-suitable (recall~\eqref{lugano}).  
By  Proposition~\ref{prp:lagstab}  it suffices to show that there does not exists a subsheaf $\cH\subset\cG$ such that $0<r(\cH)<r(\cG)$ and  
$\mu(\cH_t)=\mu(\cG_t)$ for general $t=x+y$. Suppose that such a subsheaf $\cH$ exists. Since 
$\cG_t$ is slope semistable, $\cH_t$ is slope semistable. As shown above, we have $\cH_t=\cH'_t\oplus \cH''_t$ where $\cH'_t,\cH''_t$ are   slope semistable subsheaves of the first and second addends of the decomposition in~\eqref{prodtens}
respectively, and their JH addends are a subset of the JH addends appearing in~\eqref{listaprod}. Of course the collection of JH addends is symmetric with respect to the involution exchanging the addends of the decomposition in~\eqref{prodtens}. The set of JH addends of  the first addend of the latter decomposition is in one-to-one correspondence with the points of $\Lambda(\cE_2)_y$, and the set of JH addends of  the seond addend of the latter decomposition is in one-to-one correspondence with the points of $\Lambda(\cE_2)_x$. These addends are invariant for the monodromy action. Since the decomposition curve $\Lambda(\cE_2)$ is  integral (by Corollary~\ref{crl:spettroirrid}) we get 
that  the set of JH addends of  $\cH_t$ is the same as the set of JH addends of  $\cG_t$, and hence $\rk(\cH)=\rk(\cG)$. That is a contradiction.
\end{proof}
\subsection{Proof of Theorem~\ref{thm:dadueauno}}\label{subsec:genmukmap}
Let $[\cE_2]\in \cM_{v_2}(S,h_S)$ be a general point. Then  $\cG(\cE_1,\cE_2)$ is locally free because both $\cE_1$ and $\cE_2$ are locally free, and it is  $h_{S^{[2]}}$ slope stable 
by Proposition~\ref{prp:stabonlagr}. Hence  the rational map $\varphi$ in~\eqref{modulimoduli} is defined. 
The image of $\varphi$, i.e.~the closure of the image of the 
open dense subset on which $\varphi$ is regular, 
is irreducible because $\cM_{v_2}(S,h_S)$ is irreducible: we denote it by $M_{w}(S^{[2]},h_{S^{[2]}})^{\bullet} $. By Item~(b) of Proposition~\ref{prp:coppiadie} we have $\Hom(\cE_2,\cE_1)=0$. Since 
$\cE_2$ is stable (because $[\cE_2]$ is a general point of $\cM_{v_2}(S,h_S)$), the full set of hypotheses of Lemma~\ref{lmm:vuquad} is satisfied.
It follows that $\cG(\cE_1,\cE_2)$ has unobstructed deformations, and $\Def(\cG(\cE_1,\cE_2))$ is identified with 
$\Def(\cE_2)$ via the map in~\eqref{defoprod} (recall that $\cE_1$ is spherical hence rigid). This proves that 
$M_{w_0}(S^{[2]},h_{S^{[2]}})^{\bullet} $ has the same dimension as $\cM_{v_2}(S,h_S)$ and is an irreducible component of $M_{w_0}(S^{[2]},h_{S^{[2]}}) $. 

It remains to prove that the map
\begin{equation}\label{mapparaz}
\cM_{v_2}(S,h_S)  \overset{\ov{\varphi}}{\dra}  M_{w_0}(S^{[2]},h_{S^{[2]}})^{\bullet} 
\end{equation}
defined by $\varphi$ is  birational. Since domain and codomain have the same dimension it suffices to show that $\ov{\varphi}$ is generically injective. 
Let 
$[\cE_2],[\cE'_2]\in \cM_{v_2}(S,h_S)$ be general distinct points. Then $\cE_2$, $\cE_2'$ are $h_S$ stable and, by Proposition~\ref{prp:stabonlagr}, both 
$\cG(\cE_1,\cE_2)$ and $\cG(\cE_1,\cE'_2)$ are $h_{S^{[2]}}$ stable. Since  
$\Hom(\cE_2,\cE_2')=0$, we have
\begin{equation}
\Hom(\tau_1^{*}\cE_1\otimes \tau_2^{*}\cE_2\oplus \tau_1^{*}\cE_2\otimes \tau_2^{*}\cE_1, 
\tau_1^{*}\cE_1\otimes \tau_2^{*}\cE'_2\oplus \tau_1^{*}\cE'_2\otimes \tau_2^{*}\cE_1)=0.
\end{equation}
By the BKR McKay correspondence it follows that  $\Hom(\cG(\cE_1,\cE_2),\cG(\cE_1,\cE'_2))=0$. This proves  that 
$\ov{\varphi}$ is generically injective. 
\qed
\section{The main result}\label{sec:alfinlaprova}
\subsection{Sideways nearby deformations}\label{subsec:sideforma}
\setcounter{equation}{0}
Let $X$ be a HK manifold, and let $\cF$ be a torsion free sheaf on $X$. Suppose that $\Delta(\cF)\in H^{2,2}_{\ZZ}(X)$ remains of type $(2,2)$ for all (nearby) deformations of $X$. Then $\cF$ is a modular sheaf. In fact by Remark~1.2 in~\cite{ogfascimod} it suffices to show that the orthogonal projection of $\Delta(\cF)$ to the Verbitsky subalgebra $D(X)\subset H(X)$ generated by $H^2(X)$ is a multiple of the class $q_X^{\vee}$ dual to the BBF quadratic form $q_X$. This holds because $q_X^{\vee}$ generates the subspace of classes in $D(X)\subset H^4(X)$ which 
remain of type $(2,2)$ for all (nearby) deformations of $X$. The definition below is not standard. 
\begin{dfn}
Let $X$ be a HK manifold, and let $\omega$ be a K\"ahler class on $X$. A   vector bundle  $\cF$ on $X$ is 
\emph{strongly $\omega$-projectively hyperholomorphic} if $\cF$   is  $\omega$ slope stable, $\Delta(\cF)\in H^{2,2}_{\ZZ}(X)$ remains of type $(2,2)$ for all (nearby) deformations of $X$,  and $\omega$ belongs to an open ${\mathsf a}(\cF)$-chamber 
in $\cK(X)$ (${\mathsf a}(\cF)$ is defined because $\cF$ is modular).
\end{dfn}
The definition above is motivated by Verbitsky's fundamental results in~\cite{verbitsky-hyperholo} (Theorem~2.5 and Section 11). 
\begin{prp}\label{prp:defproj}
Let $X_0$ be a HK manifold,  and let $\omega_0$ be a K\"ahler class on $X_0$.
Suppose that $\cF$ is a strongly $\omega_0$-projectively hyperholomorphic vector bundle on $X_0$.
 Then the natural maps 
\begin{equation}\label{duespadef}
\Def(X_0,\PP(\cF))\lra \Def(X_0),\qquad \Def(X_0,\cF)\lra \Def(X_0,c_1(\cF))
\end{equation}
 are surjective.
\end{prp}
\begin{proof}
Abusing notation we denote by the same symbols representatives of the deformation spaces in~\eqref{duespadef}. In particular we identify $\Def(X_0)$ and $\Def(X_0,c_1(\cF))$ with open neighborhoods of the origins in $H^{1,1}(X_0)$ and 
in $H^{1,1}(X_0)\cap c_1(\cF)^{\bot}$ respectively.
Let $\cU\subset\cK(X_0)$  be the open $a(\cF)$-chamber containing $\omega_0$. Then  
$\cF$ is $\omega$ slope stable for all 
$\omega\in \cU$ by Corollary~\ref{crl:campol}. Let $\omega\in \cU$, and let $\cX\to T(X_0,\omega)$ be the twistor family of deformations of $X_0$ 
 associated to $\omega$.  By Theorem~11.1 in~\cite{verbitsky-hyperholo} the projective bundle $\PP(\cF)\to X_0$ extends to a family of projective bundles over the fibers of $\cX\to T(X,\omega)$. 
 This proves that the image of the first map in~\eqref{duespadef} contains a neighborhood of $0$ in the open cone 
 $\cU\subset H^{1,1}_{\RR}(X_0)$; since the image
  is a complex analytic subset of  $H^{1,1}(X_0)=H^{1,1}_{\RR}(X_0)\otimes_{\RR}\CC$,   it follows that it contains a neighborhood of $0$ in 
$H^{1,1}(X_0)$. This proves that the first map in~\eqref{duespadef} is surjective. 

Next we prove that the second map in~\eqref{duespadef} is surjective. 
Let $X$ be a  (nearby) deformation of $X_0$ , and let
$g_X\colon {\bf P}_X\to X$ be a deformation of the projective bundle $g_0\colon \PP(\cF)\to X_0$ (it exists by surjectivity 
of the first map in~\eqref{duespadef}). It suffices to show that if the deformation of $X$ belongs to  $H^{1,1}(X_0)\cap c_1(\cF)^{\bot}$, i.e.~$c_1(\cF)$ remains of type $(1,1)$, then ${\bf P}_X$ is the projectivization of a vector bundle.  Let $r$ be the rank of $\cF$, and  let $\xi_0=c_1(\cO_{\PP(\cF)}(1))$. 
Then
\begin{equation}
c_1(\Theta_{\PP(\cF)/X_0})=(r+1)\xi_0+g_0^{*} c_1(\cF).
\end{equation}
This shows that if the parallel transport to $H^2(X)$ of the class $c_1(\cF)$  is of type $(1,1)$, then also the 
parallel transport to  $H^2({\bf P}_X)$ of the class $\xi_0$  is of type $(1,1)$.
If $\xi\in H^{1,1}({\bf P}_X)$ is the parallel transport of $\xi_0$, and $L$ is the corresponding holomorphic line bundle on ${\bf P}_X$, then the dual of the vector bundle $g_{X,*}(L)$ is an extension of $\cF$ to $X$ (we adopt the pre-Grothendieck convention for the projectivization of a vector bundle).
This proves that the second map in~\eqref{duespadef} is surjective
\end{proof}
\begin{prp}\label{prp:gisimuove}
Let notation and hypotheses be as in Theorem~\ref{thm:dadueauno}. Let $[\cE_2]\in\cM_{v_2}(S,h_S)$ be a general point and set $\cG=\cG(\cE_1,\cE_2)$. Then  the  map 
\begin{equation}\label{onceupon}
\Def(S^{[2]},\cG)\lra \Def(S^{[2]},c_1(\cG))
\end{equation}
is surjective.
\end{prp}
\begin{proof}
Since $[\cE_2]$ is a general point of $\cM_{v_2}(S,h_S)$ the sheaf $\cE_2$ is locally free, and hence also 
$\cG$ is locally free.  Let $h_{S^{[2]}}$ be a polarization of $S^{[2]}$ as in Theorem~\ref{thm:dadueauno}, i.e.~which is  ${\mathsf a}(w_0)$-suitable with respect to the Lagrangian fibration $\pi$ appearing in~\eqref{fiblagk3n}. 
Then $\cG$  is $h_{S^{[2]}}$ slope stable by Proposition~\ref{prp:stabonlagr},  and $\Delta(\cG)$ remains of type $(2,2)$ for all (nearby) deformations of $S^{[2]}$ because   it is a multiple of $c_2(S^{[2]})$. Now let $\omega\in\cK(S^{[2]})$ be a class belonging to an open ${\mathsf a}(\cG)$-chamber whose closure contains $h_{S^{[2]}}$. Then $\cG$ is $\omega$ slope stable by Proposition~\ref{prp:campol}. Hence $\cG$ is strongly $\omega$-projectively hyperholomorphic, and therefore the proposition follows from Proposition~\ref{prp:defproj}.
\end{proof}
\subsection{Proof of the main result}
\setcounter{equation}{0}
We recall the main result, i.e.~Theorem~\ref{thm:belteo}. The hypotheses are the following: let $r_1$ be a positive integer, 
let  $(X,h)$ be a polarized  HK variety of type $K3^{[2]}$ such that~\eqref{divdipol} and~\eqref{congruenze} hold, 
let $a$ be a positive integer greater than $1$ such that $2a$ is a multiple of $r_1$, and let $w$ be the mock Mukai vector in~\eqref{nostrodoppiovu}.  The thesis is that the moduli space $M_w(X,h)$ is non empty, and for $(X,h)$  general it has an irreducible component of  dimension $2a^2+2$. The proof proper is at the end of the subsection.

Throughout the subsection we suppose that Assumption-Definition~\ref{ass-dfn:numeretti} holds. In addition we assume the following:
\begin{equation}
\text{all singular elliptic fibers of 
$\epsilon\colon S\to\PP^1$ have simple node singularities.}
\end{equation}
Note that this holds generically, and that if it holds then each singular fiber has exactly one node  because $S$ has Picard number $2$.
Let $\cL$ be the line bundle on $S^{[2]}$ such that
\begin{equation}
c_1(\cL)=2\mu(D)-r_1\delta.
\end{equation}
\begin{prp}\label{prp:relampio}
Let $\pi$ be the Lagrangian fibration in~\eqref{fiblagk3n}. If 
$d_0\ge 2r_1$  then $\cL$ is $\pi$ ample (recall that $d_0=D\cdot C$ where $C$ is a fiber of $\epsilon$). 
\end{prp}
\begin{proof}
It suffices to show that for every $(x_1+x_2)\in(\PP^1)^{(2)}=\PP^2$
the restriction of $\cL$ to every irreducible component of the reduced fiber  $\pi^{-1}(x_1+x_2)_{\rm red}$ is ample. For $x\in\PP^1$ we let $ C_x^{[2]},W_x\subset S^{[2]}$ be the subsets parametrizing subschemes of $C_x=\epsilon^{-1}(x)$
and 
non reduced schemes $Z$ supported on  $C_x$ respectively.
Let $(x_1+x_2)\in(\PP^1)^{(2)}=\PP^2$. The irreducible  decomposition of $\pi^{-1}(x_1+x_2)_{\rm red}$ is 
\begin{equation}
\pi^{-1}(x_1+x_2)_{\rm red}=
\begin{cases}
C_{x_1}\times C_{x_2} & \text{if $x_1\not=x_2$,}\\
C_x^{[2]}\cup W_x & \text{if $x_1=x_2=x$.}
\end{cases}
\end{equation}
It follows that if $x_1\not=x_2$ then $\cL$ is ample on $\pi^{-1}(x_1+x_2)_{\rm red}$. Before continuing we introduce the following notation. Let $Z$ be a smooth curve. We let 
\begin{equation}
\NS(Z)\overset{\mu_Z}{\lra} \NS(Z^{(2)})
\end{equation}
be the map associating to (the class of) a line bundle $L$ the class of a line bundle whose pull-back to $Z^2$ via the quotient map is isomorphic to $L\boxtimes L$.  
Let us prove that $\cL$ is ample on $C_x^{[2]}$. Suppose first that $C_x$ is smooth. Then $C_x^{[2]}=C_x^{(2)}$. By hypothesis we have $2d_0= 3r_1+r_1+b$ where $b\ge 0$. Let $A$ be a divisor on $C_x$ such that $\deg A=r_1+b$, and let $L$ be a line bundle on $C_x$ of degree $3$. In $\NS(C_x^{(2)})$ we have
\begin{equation}
\cl(\cL_{C_x^{(2)}})\cong r_1(\mu_{C_x}(L)-\delta_x)+r_1\mu_{C_x}(A),\qquad \delta_x\coloneq \delta_{|C_x^{(2)}}.
\end{equation}
 We claim  that $\mu_{C_x}(L)-\delta_x$ is ample. Granting this for the moment being, it follows that $\cL$ is ample on $C_x^{(2)}$ because $\mu_{C_x}(A)$ is nef. To prove that $\mu_{C_x}(L)-\delta_x$ is ample consider the  map 
$\varphi\colon C_x^{(2)}\to |L|$ which associates to $p+q$ the unique divisor  $E\in |L|$ such that $E-p-q$ is effective. 
A straightforward computation gives that 
$c_1(\varphi^{*}\cO_{|L|}(1))=\mu_{C_x}(L)-\delta_x$. We are done because 
$\varphi$ is finite, and hence $\varphi^{*}\cO_{|L|}(1)$ is ample on $C_x^{(2)}$.

Now we prove that $\cL$ is ample on $C_x^{[2]}$ if $C_x$ is singular. Let 
\begin{equation}\label{barocco}
\PP^1\cong\wt{C}_x\overset{\alpha}{\lra} C_x
\end{equation}
be the normalization map. By hypothesis $C_x$ is nodal, with exactly one node $p$. Let $p',p''\in  \wt{C}_x$ be the two points mapped to $p$ by the normalization map. Let 
\begin{equation}
 \FF_1\cong\Blow_{p'+p''}(\wt{C}_x^{(2)})\overset{\beta}{\lra} \wt{C}_x^{(2)}\cong\PP^2
\end{equation}
be the blow up map. The normalization of $C_x^{[2]}$ is naturally isomorphic to $\Blow_{p'+p''}(\wt{C}_x^{(2)})$.
 Let $\nu\colon \FF_1\to C_x^{[2]}$  be the normalization map. It suffices to prove that $\nu^{*}(\cL_{|C_x^{[2]}})$ is ample. Let 
$\Sigma\subset\FF_1$ be the negative section of the $\PP^1$-fibration $\FF_1\to\PP^1$, and let $\Omega\subset\FF_1$ be  the inverse image via the blow up map $\FF_1\to \Blow_{p'+p''}(\wt{C}_x^{(2)})$ of the  conic  parametrizing non reduced divisors. Then, letting $\Delta\subset S^{[2]}$ be the divisor parametrizing non reduced subschemes, we have
\begin{equation}
\nu^{*}\left(\Delta_{|C_x^{[2]}}\right)=2\Sigma+\Omega.
\end{equation}
Since $2\delta=\cl(\Delta)$, it follows that
\begin{equation}
\nu^{*}\left(\cL_{|C_x^{[2]}}\right)\cong\cO_{\FF_1}((2d_0-r_1)\beta^{*}H-r_1\Sigma),
\end{equation}
where $H\subset \PP^2$ is a line.  It follows that $\nu^{*}(\cL_{|C_x^{[2]}})$ is ample. 

It remains to show that $\cL$ is ample on $W_x$. Note that we have a $\PP^1$-fibration $W_x\to C_x$, in fact 
$W_x\cong\PP(\Theta_{S|C_x})$. 
One applies the 
Kleiman-Nakai-Moishezon criterion. First one computes
\begin{equation}
c_1(\cL)^2\cdot W_x=8r_1 d_0>0
\end{equation}
by noting that the cycle $W_x$ represents the cohomology class $\mu(C)\cdot\delta$. 
Now  suppose  that $C_x$ is smooth, and let $\Gamma\subset W_x$ be an integral curve. If $\Gamma$ is a fiber of the $\PP^1$-fibration 
$W_x\to C_x$ then $\cL\cdot\Gamma=r_1>0$. If the restriction to $\Gamma$  of the $\PP^1$-fibration  $W_x\to C_x$ is dominant then 
$\delta\cdot\Gamma\le 0$ because $\delta$ is the class of the tautological (sub)line bundle on $W_x\cong\PP(\Theta_{S|C_x})$ 
and $\Theta_{S|C_x}$ is an extension of trivial line bundles on $C_x$. From this it follows that 
$\cL\cdot\Gamma\ge 2d_0>0$. This shows that $\cL$ is ample on $W_x$ if $C_x$ is smooth.

Lastly,  suppose  that $C_x$ is singular. Let $\alpha$ be the normalization map in~\eqref{barocco}, and let 
$\wt{W}_x\coloneq\PP(\alpha^{*}\Theta_{S|C_x})$. The natural map $\psi\colon \wt{W}_x\to W_x$ is the normalization. 
It suffices to prove that if 
$\Gamma\subset W_x$ is an integral curve then $\psi^{*}\cL\cdot \Gamma>0$. Note that 
$\wt{W}_x\cong\PP(\cO_{\PP^1}(2)\oplus \cO_{\PP^1}(-2))$. If $\Gamma$ is a fiber of the $\PP^1$-fibration $\wt{W}_x\to\wt{C}_x\cong\PP^1$ then 
$\psi^{*}\cL\cdot\Gamma=r_1>0$. Next suppose that the restriction to $\Gamma$  of the $\PP^1$-fibration  $W_x\to C_x$ is dominant, and let $\deg$ be its degree.  Since  $\psi^{*}(\delta_{|W_x})$ is the class of the tautological (sub)line bundle on 
$\wt{W}_x\cong\PP(\cO_{\PP^1}(2)\oplus \cO_{\PP^1}(-2))$ one gets that $\psi^{*}(\delta_{|W_x})\cdot\Gamma\le 2\deg$. On the other hand 
$\psi^{*}(2\mu(D)_{|W_x})\cdot\Gamma\ge 2d_0\deg$. Thus $\psi^{*}(\cL_{|W_x})\cdot\Gamma\ge 2(d_0-r_1)\deg>0$.
\end{proof}
Let $\pi$ be the Lagrangian fibration in~\eqref{fiblagk3n}. We let
\begin{equation}
f\coloneq\pi^{*}c_1(\cO_{\PP^2}(1))=\mu(C),
\end{equation}
where $C$ is a fiber of the elliptic fibration $\epsilon\colon S\to\PP^1$. 
Letting $i\in\{1,2\}$ be  the integer defined by the condition
$i\equiv r_1\pmod{2}$, we let
\begin{equation}
 g:=\frac{1}{i}\left(2\mu(D)-r_1\delta\right).
\end{equation}
\begin{clm}\label{clm:scorching}
Let $\Lambda\subset\NS(S^{[2]})$ be the lattice  generated by $f$ and $g$. 
\begin{enumerate}
\item
$\Lambda$  is saturated of rank $2$.
\item
$\Lambda$   contains $c_1(\cG)$ for any sheaf $\cG$ such that $w(\cG)=w_0$.
\item
The  restriction of the BBF quadratic form to $\Lambda$ has discriminant equal to $2d_0/i$.
\item
If $d_0\ge 2r_1$ then $\Lambda$ contains ample classes  which are 
${\mathsf a}(w_0)$-suitable with respect to the Lagrangian fibration $\pi$, where  $w_0$ is as in~\eqref{doppiovuzero}.
\end{enumerate}
\end{clm}
\begin{proof}
(1)  holds because $\NS(S)$ is freely generated by the classes $[C]$ and $[D]$. (2) holds  by definition of $w_0$, see~\eqref{doppiovuzero}. A straightforward computation gives~(3). 
We prove~(4). By Proposition~\ref{prp:relampio} the class $g$ is $\pi$ ample, hence $g+Nf$ is ample for $N\gg 0$. By Corollary~\ref{crl:suitif} we get that for an $M\gg N$  the class $g+Mf$ is ample and ${\mathsf a}(w_0)$-suitable. 
\end{proof}
From now on we assume that $d_0\ge 2r_1$. Let $\cX(\Lambda)\to T(\Lambda)$ be a representative of the deformation space $\Def(S^{[2]},\Lambda)$ of deformations of $S^{[2]}$ which keep the classes in $\Lambda$ of type $(1,1)$.  
For $t\in T(\Lambda)$ we let $X_t$ be the corresponding HK variety. 
We let $X_0=S^{[2]}$. We assume that monodromy acts trivially on  classes in $\Lambda$. For $t\in T(\Lambda)$ we let $\Lambda_t\subset\NS(X_t)$ be the rank $2$ lattice obtained from $\Lambda$ by parallel transport.
For $t\in T(\Lambda)$ we let $w_t$ be the mock Mukai vector for $X_t$ obtained from $w_0$ by parallel transport.

Choose an ${\mathsf a}(w_0)$-suitable polarization $p_0\in\Lambda_0=\Lambda$, and for $t\in T(\Lambda)$ let $p_t\in\Lambda_t$ be the corresponding ${\mathsf a}(w_0)$-suitable polarization. Similarly, for $t\in T(\Lambda)$ let $g_t\in\Lambda_t$ be the class obtained from $g$ by parallel transport. 
By Maruyama there exists a scheme of finite type $M(\cX(\Lambda)/T(\Lambda))$ and a regular map 
$M(\cX(\Lambda)/T(\Lambda))\to T(\Lambda)$ whose fiber over $t$ is isomorphic to the moduli space $M_{w_t}(X_t,p_t)$. (Recall Item~(2) of Claim~\ref{clm:scorching}.)
\begin{clm}\label{clm:compgrandim}
There exists an open dense subset $\cU(\Lambda)\subset T(\Lambda)$ such that for $t\in \cU(\Lambda)$ the moduli space $M_{w_t}(X_t,p_t)$ has an irreducible component $M_{w_t}(X_t,p_t)^{\bullet}$ of dimension $2a^2+2$.
\end{clm}
\begin{proof}
Follows from Theorem~\ref{thm:dadueauno} and Proposition~\ref{prp:gisimuove}.
\end{proof}
\begin{clm}\label{clm:moltogen}
Assume that $4m_0-r_1^2>0$. Then there exists $b(m_0)$ such that the following holds. Suppose that $D\cdot D=2m_0$, 
$D\cdot C=d_0>b(m_0)$, and  $t\in T(\Lambda)$ is such that $\NS(X_t)=\Lambda_t$. Then  $g_t$ is ample, $w_t$-generic, and   the open $w_t$-chamber containing it contains also $p_t$.
\end{clm}
\begin{proof}
The lattice  $\Lambda$ is generated by $f$ and $g$, and $q_{S^{[2]}}(g)=(8m_0-2r_1^2)/i^2>0$. Let $\gamma\in\Lambda$ be such that $q(\gamma)<0$. Then 
\begin{equation}\label{belladiseg}
q_{S^{[2]}}(\gamma)\le -\frac{4d_0^2}{i^2+(8m_0-2r_1^2)}
\end{equation}
by Lemma~4.3 in~\cite{ogfascimod}. In particular if $d_0\gg 0$ we have $q_{S^{[2]}}(\gamma)<-10$. Since $q_{S^{[2]}}(g_t)>0$ it follows that one of $g_t$, $-g_t$ is ample, and then it has to be $g_t$ (note: we are showing that  if $d_0\gg 0$ the ample cone coincides with the positive cone). Similarly, if $d_0\gg 0$ we have $q_{S^{[2]}}(\gamma)<-10a^4 r_18$, and it follows  that there is a single open $w_t$-chamber 
(see Example~\ref{expl:adigi}).
\end{proof}
 Now assume that $4m_0-r_1^2>0$, and that $d_0>b(m_0)$ (of course $d_0\ge 2r_1$). Let $t_{*}\in \cU(\Lambda)$ be such that 
 $\NS(X_{t_{*}})=\Lambda_{t_{*}}$. By Claim~\ref{clm:moltogen} the class $g_{t_{*}}$ is ample, $w_{t_{*}}$-generic, and the moduli space $M_{w_{t_{*}}}(X_{t_{*}},g_{t_{*}})$ is isomorphic to $M_{w_{t_{*}}}(X_{t_{*}},p_{t_{*}})$. By Claim~\ref{clm:compgrandim} it follows that $M_{w_{t_{*}}}(X_{t_{*}},g_{t_{*}})$ has an irreducible component 
 $M_{w_{t_{*}}}(X_t,g_{t_{*}})^{\bullet}$ of dimension $2a^2+2$. 
 
 Let $\cY(g_{t_{*}})\to T(g_{t_{*}})$ be a complete family of polarized HK varieties of type $K3^{[2]}$ with irreducible parameter space, containing a polarized couple isomorphic to $(X_{t_{*}},g_{t_{*}})$. 
 For $t\in T(g_{t_{*}})$ we let $(Y_t,g_t)$ be the corresponding polarized HK of type $K3^{[2]}$, and we let  $w_t$ be the mock Mukai vector for $Y_t$ obtained from $w_0$ by parallel transport.
By Maruyama there exist a scheme of finite type $M(\cY(g_{t_{*}})/T(g_{t_{*}}))$ and a regular map 
$M(\cY(g_{t_{*}})/T(g_{t_{*}}))\to T(g_{t_{*}})$ whose fiber over $t$ is isomorphic to the moduli space $M_{w_t}(Y_t,g_t)$. 
 By Proposition~\ref{prp:defproj} we get that for $t\in  T(g_{t_{*}})$ general the moduli space $M_{w_t}(Y_t,g_t)$ contains an irreducible component of dimension $2a^2+2$. 
 
Next note that 
\begin{equation}
\divisore(g_t)=
\begin{cases}
1 & \text{if $r_1\equiv 0\pmod{2}$, i.e.~$i=2$,} \\
2 & \text{if $r_1\equiv 1\pmod{2}$, i.e.~$i=1$,} 
\end{cases}
\end{equation}
and
\begin{equation}
q_{Y_t}(g_t)=
\begin{cases}
2m_0-\frac{r_1^2}{2} & \text{if $r_1\equiv 0\pmod{2}$, i.e.~$i=2$,} \\
8m_0-2r_1^2 & \text{if $r_1\equiv 1\pmod{2}$, i.e.~$i=1$.} 
\end{cases}
\end{equation}
Now recall that $r_1\mid(m_0+1)$ (see~\eqref{condizioni}). It follows that~\eqref{divdipol} and~\eqref{congruenze} hold and that, conversely, if ~\eqref{divdipol} and~\eqref{congruenze} hold  
then $(X,h)$ is isomorphic to $(Y_t,g_t)$ for a suitable $t\in T(g_{t_{*}})$. This finishes the proof of Theorem~\ref{thm:belteo}.

\appendix

\section{Sheaves on elliptic K3 surfaces}\label{sec:sheavesonellk3}
\subsection{Outline of the section}  
\setcounter{equation}{0}
Let $S$ be a  $K3$  surface with  an elliptic fibration $S\to\PP^1$. Let $\cF$ be a torsion free sheaf on $S$, and let $h_S$ be an ample divisor on $S$ which is ${\mathsf a}(\cF)$-suitable. Suppose that $\cF$ is $h_S$ slope semistable. Let $C_x$ be the elliptic fiber over $x\in\PP^1$. If $x$ is general then $\cF_x=\cF_{|C_x}$ is slope semistable by 
Proposition~\ref{prp:ellstab}. The  graded vector bundle associated to the Jordan-H\"older filtration of $\cF_x$ is isomorphic to a direct sum $E_1(x)\oplus \ldots\oplus E_a(x)$ of vector bundles with equal ranks and degrees. 
Taking the determinants of the direct factors, and letting $x$ vary in $\PP^1$, one gets a curve $\Lambda(\cF)$ in a suitable Jacobian fibration 
$J^d(S/\PP^1)$. By associating to $[\cF]\in \cM_v(S,H_S)$ (here $v$ is a  Mukai vector, and $h_S$ is $v$-suitable) the curve $\Lambda(\cF)$, one gets a regular map from $\cM_v(S,H_S)$ to a linear system on $J^d(S/\PP^1)$. We show that, under certain hypotheses, this map is surjective. 
\subsection{The decomposition curve}\label{subsec:curvaspett}
\setcounter{equation}{0}
Let $S$ be an elliptic $K3$ surface with an elliptic fibration 
$\varepsilon\colon S\to \PP^1$ 
as in Claim~\ref{clm:neronsevero}. We recall that this means that
\begin{equation}
\NS(S)=\ZZ[D]\oplus\ZZ[C], \quad
D\cdot D=2m_0,\quad D\cdot C=d_0,
\end{equation}
where $m_0,d_0$ are positive integers and $C$ is an elliptic fiber.
 Let $u\coloneq (0,C,d_0)$, and let $J^{d_0}(S/\PP^1):=\cM_u(S,H_S)$ be the relative degree-$d_0$ Jacobian of $S\to\PP^1$. By Mukai's well-known results
$J^{d_0}(S/\PP^1)$ is a $K3$ surface. Moreover 
there is the regular map $J^{d_0}(S/\PP^1)\to\PP^1$ associating to $[\xi]\in J^{d_0}(S/\PP^1)$  the point $x\in\PP^1$ such that 
$C_x\coloneq \varepsilon^{-1}(x)$ is the  support of the sheaf $\xi$.
This is an elliptic fibration with the section  which associates to $x\in\PP^1$ the class of the restriction of $\cO_S(D)$ to $C_x$. Hence $J^{d_0}(S/\PP^1)$ is a $K3$  elliptic surface with a section. Let $\Gamma$ be an elliptic fiber of 
$J^{d_0}(S/\PP^1)\to\PP^1$, and let $\Sigma$ be the image of the section defined above. We have
\begin{equation}
\Sigma\cdot\Sigma=-2,\quad \Sigma\cdot\Gamma=1,\quad \Gamma\cdot\Gamma=0.
\end{equation}
Now suppose that $r_1$ is a positive integer as in Assumption-Definition~\ref {ass-dfn:numeretti}, and let $v_1$ be as in loc.~cit., i.e.
\begin{equation}
v_1\coloneq \left(r_1,D,\frac{m_0+1}{r_1}\right).
\end{equation}
Let $a,b\in\NN$ with $a>0$, and let
\begin{equation}\label{eccovudue}
v_2\coloneq av_1-b(0,0,1).
\end{equation}
Let $\cF$ be a torsion free sheaf on $S$ with $v(\cF)=v_2$. We define the associated decomposition curve in 
$J^{d_0}(S/\PP^1)$ as follows. Let $\cM_{r_1,d_0}(S/\PP^1)\to\PP^1$ be the relative (Simpson) moduli space parametrizing 
semistable sheaves on fibers $C_x$ of rank $r_1$ and degree $d_0$ (they are all stable because $\gcd\{r_1,d_0\}=1$). We have an isomorphism
\begin{equation}\label{mappadet}
\begin{matrix}
\cM_{r_1,d_0}(S/\PP^1) & \overset{\sim}{\lra} & J^{d_0}(S/\PP^1) \\
[\cG] & \mapsto & [\det\cG]
\end{matrix}
\end{equation}
\begin{rmk}
Let $h_S$ be an ample divisor on $S$ which is $v_2$-suitable (with respect to the elliptic fibration 
$\epsilon\colon S\to\PP^1$). By Proposition~\ref{prp:coppiadie} there exists a unique $h_S$-stable sheaf $\cE_1$ on $S$ such that $v(\cE_1)=v_1$, and the restriction of $\cE_1$ to every elliptic fiber is stable. Thus $\cE_1$ determines a section  $\sigma\colon\PP^1\to \cM_{r_1,d_0}(S/\PP^1)$.  The image of $\sigma$ is equal to $\Sigma$ because 
$\det(\cE_1)\cong\cO_S(D)$.
\end{rmk}

Let $U\subset\PP^1$ be a sufficiently small open subset in the classical topology, and let $S_U:=f^{-1}(U)$. Then there exists a tautological sheaf $G_U$ on $S_U\times_{\PP^1}\cM_{r_1,d_0}(S/\PP^1)$. For $x\in \PP^1$ and 
$[\cG]\in \cM_{r_1,d_0}(C_x)$ we have $\chi_{\cO_{C_x}}(\cG,\cF_{|C_x})=0$. It follows that there exists a determinant line bundle $\cL_U$  on $\cM_{r_1,d_0}(S_U)$ and a section $s_U\in\Gamma(\cM_{r_1,d_0}(S_U),\cL_U)$ whose zero-scheme  $Z(s_U)$ is supported on the set
\begin{equation*}
\{[\cG]\in \cM_{r_1,d_0}(S_U) \mid \text{$\cG$ supported on $C_x$ and $\Hom(\cG_x,\cF_{|C_x})\not=0$}\}.
\end{equation*}
 The line bundles $\cL_U$ and sections $s_U$ for varying $U$ glue to give a line bundle $\cL(\cF)$ and a section $s(\cF)$ on 
$\cM_{r_1,d_0}(S/\PP^1)$. We let $\Lambda(\cF)\subset\cM_{r_1,d_0}(S/\PP^1)$ be the zero-scheme of $s(\cF)$. Via the identification in~\eqref{mappadet} we view $\Lambda(\cF)$ as a subscheme of $J^{d_0}(S/\PP^1)$. Restricting the elliptic fibration $J^{d_0}(S/\PP^1)\to\PP^1$ to  $\Lambda(\cF)$ we get a regular map  $\Lambda(\cF)\to\PP^1$; we let $\Lambda(\cF)_x$ be the fiber of this map over $x\in\PP^1$. 
\begin{rmk}\label{rmk:bellinzona}
Let $T$ be irreducible and let $\cF\to S\times T$ be a $T$-flat family of sheaves as above, i.e.~for all $t\in T$ we have $v(\cF_{|S\times\{t\}})=v_2$. Then the isomorphism class of the line bundle $\cL(\cF_{|S\times\{t\}})$ is independent of $t\in T$. Since the moduli space $\cM_{v_2}(S,h_S)$ is irreducible it follows that the isomorphism class of $\cL(\cF)$ is independent of the point $[\cF]\in \cM_{v_2}(S,h_S)$.
We let $\cL(v_2)\coloneq\cL(\cF)$ for any $[\cF]\in \cM_{v_2}(S,h_S)$.  
\end{rmk}
\begin{rmk}\label{rmk:seadatto}
Suppose that the restriction of $\cF$ to a general elliptic fiber  is  semistable.  Then $\Lambda(\cF)$ is a curve. Let $x\in\PP^1$ be such that 
$\cF_{|C_x}$ is semistable. 
Since $r_1$ and $d_0$ are coprime,  the associated graded vector bundle of $\cF_{|C_x}$ is isomorphic to 
$V_1(x)\oplus \ldots\oplus V_a(x)$,
where $r(V_i(x))=r_1$ and $\deg(V_i(x))=d_0$ for all $i\in\{1,\ldots,a\}$. Then, identifying codimension $1$ subschemes of $C_x$ with effective divisors, we have
\begin{equation}
\Lambda(\cF)_x=[\det V_1(x)]+ \ldots + [\det V_a(x)].
\end{equation}
If $\cF_{|C_x}$ is not semistable, then $C_x$ appears in  $\Lambda(\cF)$ (we identify $\Lambda(\cF)$ with an effective divisor on 
$J^{d_0}(S/\PP^1)$) with positive multiplicity.
\end{rmk}
\subsection{Decomposition curves and Lagrangian fibrations} 
\setcounter{equation}{0}
Below is the main result of the present section.
\begin{prp}\label{prp:mappaspettrale}
Keep notation and hypotheses as above, in particular $v_2$ is given by~\eqref{eccovudue}. Then the following hold:
\begin{enumerate}
\item
We have an isomorphism
\begin{equation}\label{ellevudue}
\cL(v_2)\cong \cO_{J^{d_0}(S/\PP^1)}(a\Sigma+br_1 \Gamma).
\end{equation}
\item
Suppose  that  $br_1\ge 2a$, and that 
$h_S$ is an ample divisor on $S$ which is $v_2$-suitable. Then the map (see Item~(1) and Remark~\ref{rmk:seadatto})
\begin{equation}\label{mappaspettrale}
\begin{matrix}
\cM_{v_2}(S,H_S) & \lra & |\cO_{J^{d_0}(S/\PP^1)}(a\Sigma+br_1  \Gamma)| \\
[\cF] & \mapsto & \Lambda(\cF)
\end{matrix}
\end{equation}
is a Lagrangian fibration.
\end{enumerate}
\end{prp}
Before proving Proposition~\ref{prp:mappaspettrale} we go through a preliminary result. Let $X$ be a smooth (irreducible) surface, and let 
$\pi\colon X\to T$ a projective map to a smooth curve. Let $\cE$ be a torsion free sheaf on $X$ with the property that for all $t\in T$ we have $\chi(X_t,E_t)=0$, where $X_t:=\pi^{-1}(t)$, and $E_t:=\cE_{|X_t}$. Then there exists a determinant line bundle $\cL(\cE)$ on $T$ and a section $s(\cE)$ of $\cL(\cE)$ whose zero-scheme $Z(s(\cE))$ is supported  on the set of $t$ such that $h^0(X_t,E_t)=h^1(X_t,E_t)>0$. Note that the double dual $\cE^{\vee\vee}$ satisfies the same hypotheses as $\cE$, hence we have  a determinant line bundle $\cL(\cE^{\vee\vee})$ on $T$ and a section $s(\cE^{\vee\vee})$ of 
$\cL(\cE^{\vee\vee})$. We let $Q(\cE):=\cE^{\vee\vee}/\cE$, where $\cE\hra\cE^{\vee\vee}$ by the canonical (injective) map. 
\begin{lmm}\label{lmm:doppioduale}
Keep hypotheses as above, and assume in addition that $h^0(X_t,\cE^{\vee\vee}_{|X_t})=h^1(X_t,\cE^{\vee\vee}_{|X_t})=0$ for all $t\in T$, and hence $Z(s(\cE^{\vee\vee}))=0$ (we identify codimension $1$ subschemes of $T$ with effective divisors). Then 
\begin{equation}
Z(s(\cE))=\sum\limits_{p\in\supp Q(\cE)}\ell(\cO_{Q(\cE),p})\pi(p).
\end{equation}
\end{lmm}
\begin{proof}
Let $A$ be a relative effective divisor on $X$ such that $h^1(X_t,E_t\otimes(A_t))=0$ for all $t\in T$, where $A_t:=A_{|X_t}$. Assume also that the supports of $A$ and $Q(\cE)$ are disjoint. Consider the  commutative diagram of sheaves on $T$:
\begin{equation}
\xymatrix{ \pi_{*}(\cE\otimes\cO_X(A))\ar[d]_{\alpha}\ar[rr]^{\gamma}    &  &  \pi_{*}(\cE\otimes\cO_X(A)_{|A}) \ar[d]^{\beta}\\ 
\pi_{*}(\cE^{\vee\vee}\otimes\cO_X(A)) \ar[rr]^{\delta} & & \pi_{*}(\cE^{\vee\vee}\otimes\cO_X(A)_{|A})}
\end{equation}
All sheaves in the above diagram are locally free of the same rank, and $Z(s(\cE))$ is the zero-scheme of the determinant of $\gamma$. The map $\beta$ is an isomorphism (the supports of $A$ and $Q(\cE)$ are disjoint), and  $\delta$ is an isomorphism because $h^0(X_t,\cE^{\vee\vee}_{|X_t})=h^1(X_t,\cE^{\vee\vee}_{|X_t})=0$ for all $t\in T$ by hypothesis. It follows that $Z(s(\cE))$ equals the zero-scheme of $\det\alpha$. 
The lemma follows because we have an exact sequence
\begin{equation}
0\lra  \pi_{*}(\cE\otimes\cO_X(A))\overset{\alpha}{\lra} \pi_{*}(\cE^{\vee\vee}\otimes\cO_X(A))\lra \pi_{*}(Q(\cE))\lra 0.
\end{equation}
\end{proof}
\begin{proof}[Proof of Proposition~\ref{prp:mappaspettrale}]
(1): Since $v_1^2=-2$, the moduli space $\cM_{v_1}(S,H_S)$ is a singleton parametrizing a vector bundle $\cE$ on $S$ whose restriction to every elliptic fiber $C_x$ is the unique stable vector bundle of rank $r_1$ and determinant isomorphic to $\cO_{C_x}(D)$. It follows that 
\begin{equation}
\cM_{(ar_1,aD,as_1)}(S,h_S)=\{[\cE^{\oplus a}]\}.
\end{equation}
Clearly $\Lambda(\cE^{\oplus a})=a\Sigma$, and this proves the validity of~\eqref{ellevudue} if $b=0$. Now assume that $b>0$. Let $\cF$ be a sheaf on $S$ fitting into the exact sequence
\begin{equation}
0\lra\cF\lra \cE^{\oplus a} \overset{\phi}{\lra} \CC_{y_1}\oplus\ldots\oplus\CC_{y_b}\lra 0,
\end{equation}
where $y_1,\ldots,y_b\in S$ are general points, and $\phi$ is a general morphism. Let $x_i=f(y_i)$. Then $\cF$ is an $h_S$ slope stable torsion-free sheaf, and $v(\cF)=v_2$. It is clear that there exists $m_1,\ldots,m_b\in\NN$ such that
\begin{equation}\label{spettrodising}
\Lambda(\cF)=a\Sigma+m_1 \Gamma_{x_1}+\ldots+ m_b \Gamma_{x_b}.
\end{equation}
Let  $i\in\{1,\ldots,b\}$. Since $\cF_{|C_{y_i}}\cong \ov{\cF}_i\oplus\CC_{y_i}$, where $\ov{\cF}_i$ is a subsheaf of $(\cE_{|C_{y_i}})^{\oplus a}$, we have  $\dim\Hom(\cG,\cF_{|C_{y_i}})= r_1$ for $[\cG]\in (\cM_{r_1,d_0}(C_{y_i})\setminus\{[\cE_{|C_{y_i}}]\})$. Thus $m_i\ge r_1$.
One proves that equality holds by applying Lemma~\ref{lmm:doppioduale}. In fact, let $T\subset \cM_{r_1,d_0}(S/\PP^1)$ be a (non projective) smooth irreducible  curve with the following properties: $T$ meets $\cM_{r_1,d_0}(C_{y_i})$ at a single point $[\cG]\not=[\cE_{|C_{y_i}}]$ and the intersection is transverse, all sheaves  parametrized by $T$ are push-forwards of \emph{locally free} sheaves on curves of the elliptic fibration $f\colon S\to\PP^1$ (i.e.~$T$ does not meet the critical set of $f$), and the surface  $X:=S\times_{\PP^1}T$ is smooth. Let $\rho\colon X\to S$ and $\pi\colon X\to T$  be the projections.  On $X$ we have the sheaf $\rho^{*}\cF$,   and a tautological locally-free sheaf $\cA$ with the property that $\cA{|p_T^{-1}(t)}$ is isomorphic to the vector bundle on $X_t\cong C_t$ corresponding to $t$ (there exists such a tautological sheaf because $H^2(T,\cO_T^{*})=0$).
 The pull-back of the determinant line bundle $\cL(F)$ to $T$ is isomorphic to the determinant line bundle $\cL(\cA^{\vee}\otimes \rho^{*}\cF)$.  The hypotheses of Lemma~\ref{lmm:doppioduale}  are satisfied by the sheaf 
 $\cE:=\cA^{\vee}\otimes \rho^{*}\cF$  on the smooth surface $X$. By that lemma we get that the canonical section of $\cL(\cA^{\vee}\otimes \rho^{*}\cF)$ vanishes at $y_i$ with multiplicity $r_1$, and hence $m_i=r_1$. This proves that 
 $\cL(\cF)\cong \cO_{J^{d_0}(S/\PP^1)}(a\Sigma+br_1 \Gamma)$.
 
(2):  Let $J^{d_0}:=J^{d_0}(S/\PP^1)$. Then $H^p(J^{d_0},\cO_{J^{d_0}}(a\Sigma+br_1 \Gamma)=0$ for $p>0$ because of the hypothesis $br_1\ge 2a$ (for example because   $a\Sigma+br_1 \Gamma$ is big and nef).
By Hirzebruch-Riemann-Roch it follows that
\begin{equation*}
\dim|\cL(v_2)|=\dim|a\Sigma+br_1 \Gamma|=1+abr_1-a^2. 
\end{equation*}
The map in~\eqref{mappaspettrale} is not constant because all the curves appearing in~\eqref{spettrodising} belong to the image. By Matsushita's Theorem  the image of the map has dimension equal to
\begin{equation}
\frac{1}{2}\dim \cM_{v_2}(S,h_S)=\frac{1}{2}(2+v_2^2)=1+abr_1-a^2. 
\end{equation}
This finishes the proof of (2).
\end{proof}
\begin{prp}\label{prp:spettroirrid}
Let $n\ge 2m\ge 2$ and let $A\in |m\Sigma+n\Gamma|$ be a general divisor.   Let
$A=A_{\rm hor}+A_{\rm vert}$
be the unique decomposition into effective divisors such that $A_{\rm vert}$ is a sum of elliptic fibers and the support of $A_{\rm hor}$ contains
 no elliptic fiber. Then 
$A_{\rm hor}$ is an integral divisor. If $m\ge 2$, then $A$ itself is an integral smooth divisor.
\end{prp}
\begin{proof}
 We proceed by induction on $m$. 
If $m=1$ the statement is trivially true because  $A_{\rm hor}=\Sigma$ for any $A\in |\Sigma+n\Gamma|$. Now 
assume that  $n\ge 2m\ge 4$. We claim that
\begin{equation}\label{svanimento}
H^1(J^{d_0},\cO_{J^{d_0}}((m-1)\Sigma+n\Gamma))=0.
\end{equation}
  In fact let $B\in|(m-1)\Sigma+n\Gamma|$ be general. Since 
$B_{\rm hor}$ is an integral divisor, the divisor $B$ is connected, i.e.~$h^0(B,\cO_B)=1$. It follows that 
$H^1(J^{d_0},\cO_{J^{d_0}}(-B))=0$, and  by 
Serre duality we get the vanishing in~\eqref{svanimento}. The restriction of $\cO_S(m\Sigma+n\Gamma)$ 
to $\Sigma$ has non negative degree because $n\ge 2m\ge 4$, and hence (the restriction) has non zero sections 
because $\Sigma$ is rational. By the vanishing in~\eqref{svanimento} it follows that 
 $|m\Sigma+n\Gamma|$ is globally generated at every point of $\Sigma$. Since  $|m\Sigma+n\Gamma|$ is  clearly 
 globally generated  away from $\Sigma$, it follows that it is globally generated. Let $A\in|m\Sigma+n\Gamma|$ be general. Then $A$ is smooth because  $|m\Sigma+n\Gamma|$ is globally generated. We claim that $A$ is irreducible (and hence integral). Suppose the contrary. Then $A=A_1+A_2$ where $A_i\in|m_i\Sigma+n_i\Gamma|$ are (non zero) smooth divisors. Since $A$ is smooth it follows that $A_1\cdot A_2=0$. 
 This leads to a contradiction. In fact it implies right away that $m_1>0$ and $m_2>0$, and since
  $n_i\ge 2m_i> 0$ we get that $A_1\cdot A_2\ge 2m_1 m_2> 0$. 
\end{proof}
The above proposition gives the following result.  
\begin{crl}\label{crl:spettroirrid}
Let hypotheses be as in Item~(2) of Proposition~\ref{prp:mappaspettrale}, in particular $v_2$ is given by~\eqref{eccovudue} and $br_1\ge 2a$. 
 If  $a\ge 2$ and $[\cF]\in M_{v_2}(S,h_S)$ is a general point, then $\Lambda(\cF)$ is an integral and smooth divisor. 
\end{crl}
%


\begin{thebibliography}{O'G22b}


\bibitem[Bec23]{beckmann-ext-muk} Thorsten Beckmann, \emph{Derived categories of Hyper-K\"ahler manifolds: extended
  Mukai vector and integral structure}, Compos.~Math. \textbf{159} (2023), 109--152.

\bibitem[Bec25]{beckmann-atomic}
Thorsten Beckmann, \emph{Atomic objects on hyper-K\"ahler manifolds},  J.~Algebraic Geometry \textbf{34} (2025), 109--160.

\bibitem[Bot24a]{bottini-og10}
Alessio Bottini, \emph{Towards a modular construction of OG10},
 Compos.~Math. \textbf{160} (2024), 2496--2529.
 
 \bibitem[Bot24b]{bottini-og10-reloaded}
Alessio Bottini, \emph{O'Grady's tenfolds from stable bundles on hyper-K\"ahler fourfolds}, 
  arXiv:2411.18528  [math.AG]

Authors: Alessio Bottini
 
\bibitem[BKR01]{bkr}
Tom Bridgeland, Alistarir King, Miles Reid\emph{The McKay correspondence as an equivalence of derived categories},
 J.~of the AMS \textbf{14} (2001), 535--554.


  
  \bibitem[Fat24]{fatighenti-esempi}
Enrico Fatighenti, \emph{Examples of Non-Rigid, Modular Vector Bundles on Hyperk\"ahler Manifolds},
Int.~Math.~Res.~Not.~IMRN  (2024), 8782--8793. 



\bibitem[GT]{greb-toma-cmpct-mod-sheaves-arxiv}
Daniel Greb, Matei Toma, \emph{Compact moduli spaces for slope-semistable
  sheaves}, arXiv:1303.2480 [math.AG].

\bibitem[GT17]{greb-toma-cmpct-mod-sheaves}
Daniel Greb, Matei Toma, \emph{Compact moduli spaces for slope-semistable
  sheaves}, Algebr. Geom. \textbf{4} (2017),  40--78. 

\bibitem[Hai01]{haiman}
Mark Haiman, \emph{Hilbert schemes, polygraphs and the Macdonald positivity
  conjecture}, J. Amer. Math. Soc. \textbf{14} (2001),  941--1006.
  
 
 
 
\bibitem[HL10]{huy-lehn-book}
Daniel Huybrechts and Manfred Lehn, \emph{The geometry of moduli spaces of sheaves}, Cambridge Math. Lib. CUPress, Cambridge (2010).
 
\bibitem[K18]{krug-bkr} 
Andreas Krug, \emph{Remarks on the derived McKay correspondence for Hilbert schemes of points and tautological bundles}, Math. Ann. \textbf{371} (2018), 461--486. 
 
\bibitem[LZ+17]{lizhangzhang} 
Jiayu Li, Chuanjing Zhang, Xi Zhang, \emph{Semi-stable Higgs sheaves and Bogomolov type inequality}, 
Calc.~Var.~Partial Differential Equations \textbf{56} (2017),  Paper No.~81.

\bibitem[Mar24a]{markman-1-obstructed}
Eyal Markman, \emph{Stable vector bundles on a hyper-k\"ahler manifold with a
  rank 1 obstruction map are modular}, Kyoto J. Math. \textbf{64} (2024),  635--742.
  
\bibitem[Mar24b]{markman-rat-isometries}
Eyal Markman, \emph{Rational Hodge isometries of Hyper-K\"ahler varieties of $K3^{[n]}$ type are algebraic}, 
Compos.~Math.,  online  DOI: https://doi.org/10.1112/S0010437X24007048
  
\bibitem[Mat86]{matsumura} Hideyuki Matsumura, \emph{Commutative rigng theory}, 
Cambridge Studies in Advanced Mathematics \textbf{8} (1986), CUP.
  
\bibitem[O'G22]{ogfascimod}
Kieran~G. O'Grady, \emph{Modular sheaves on hyperk{\"a}hler varieties}, Algebr.
  Geom. \textbf{9} (2022), 1--38.

\bibitem[O'G24]{og-rigidi-su-k3n}
Kieran~G. O'Grady, \emph{Rigid stable vector bundles on hyperk{\"a}hler
  varieties of type ${K}3^{[n]}$},  
  J.~Inst.~Math.~Jussieu \textbf{23} (2024), 2051--2080.


\bibitem[O'G25]{gen-hk-as-mod-sheaves}
Kieran~G. O'Grady, \emph{General polarized HK varieties as moduli spaces of sheaves},  
 work in progress.


\bibitem[Tae19]{taelman-der-hks}
Lenny Taelman, \emph{Derived equivalences of hyperk\"ahler varieties},
 Geom.~Topol. \textbf{27} (2023),  2649--2693.

\bibitem[Ver96]{verbitsky-hyperholo}
Mikhail Verbitsky, \emph{Hyperholomorphic bundles over a hyper-{K}\"{a}hler
  manifold}, J.~Algebraic Geom. \textbf{5} (1996),  633--669.
 
\bibitem[Yos99]{yoshi-ellsurf} 
K$\bar{\text o}$ta Yoshioka, 
\emph{Some notes on the moduli of stable sheaves on elliptic surfaces}, 
Nagoya Math.~J. \textbf{154} (1999), 73--102.

\bibitem[Wie16]{wieneck1}
Benjamin Wieneck, \emph{On polarization types of {L}agrangian fibrations},
  Manuscripta Math. \textbf{151} (2016),  305--327. 

\end{thebibliography}

 \end{document}